\documentclass[a4paper,10pt]{article}
\usepackage{amsmath,amsfonts,amssymb,amsthm,amsbsy,enumerate,thmtools,verbatim}
\usepackage{fullpage}
\usepackage{hyperref,aliascnt}

\makeatletter
\def\tagform@#1{\maketag@@@{\ignorespaces#1\unskip\@@italiccorr}}
\let\orgtheequation\theequation
\def\theequation{(\orgtheequation)}
\makeatother
\DeclareMathOperator{\Span}{\textrm{Span}}
\providecommand{\cF}{\mathcal{F}}
\providecommand{\cG}{\mathcal{G}}
\providecommand{\cH}{\mathcal{H}}
\DeclareMathOperator*{\EE}{\mathbb{E}}
\providecommand{\VV}{\mathbb{V}}
\providecommand{\RR}{\mathbb{R}}
\providecommand{\normal}{\mathcal{N}}
\providecommand{\rand}{\mathcal{R}}
\providecommand{\pset}[1]{2^{#1}}
\providecommand{\randx}{U(\pset{[n]})}
\providecommand{\charf}[1]{\chi_{#1}}
\providecommand{\cond}[1]{\mathbf{1}_{#1}}
\providecommand{\eqdef}{\triangleq}
\providecommand{\upto}[2]{#1 \pm #2}
\providecommand{\error}{\epsilon_1}
\providecommand{\errub}{\epsilon_0}
\providecommand{\errorh}{\epsilon_2}

\providecommand{\canonicalexample}
{\vspace{-10pt} \[
\begin{aligned}
& \mspace{10mu} \begin{array}{c}
\overbrace{\hphantom{1-\frac{1}{n}\kern2\tabcolsep\cdots\kern1\tabcolsep1-\frac{1}{n}}}^{\frac{n}{2}}
\end{array} \\[-8pt]
&\begin{pmatrix}

1-\frac{1}{n} & \cdots & 1 - \frac{1}{n} & \frac{1}{n}- 1 & \cdots & \frac{1}{n} -1 \\
-\frac{1}{n} & \cdots &- \frac{1}{n} & \frac{1}{n} & \cdots & \frac{1}{n} \\
-\frac{1}{n} & \cdots & -\frac{1}{n} & \frac{1}{n} & \cdots & \frac{1}{n}\\
\vdots && \vdots & \vdots && \vdots \\
-\frac{1}{n} & \cdots & -\frac{1}{n} & \frac{1}{n} & \cdots & \frac{1}{n}
\end{pmatrix}
\end{aligned}
\]}

\newtheorem{theorem}{Theorem}
\newtheorem{conjecture}{Conjecture}
\newtheorem{lemma}{Lemma}
\newtheorem*{classic}{Theorem}
\newtheorem*{remark}{Remark}
\newaliascnt{lemmaa}{lemma}
\newtheorem{corollary}[lemmaa]{Corollary}
\newcommand{\sgn}{\textrm{sgn}}

\newcommand{\remove}[1]{}

\title{A stability result for balanced dictatorships in $S_n$}
\author{David Ellis\footnote{Queen Mary, University of London, UK.}, Yuval Filmus\footnote{University of Toronto, Canada. Research supported by the Canadian Friends of the Hebrew University / University of Toronto Permanent Endowment.}, and Ehud Friedgut\footnote{Weizmann Institute of Science, Israel. Research supported in part by I.S.F. grant 0398246, and BSF grant 2010247.}}
\date{May 2013}
\begin{document}

\maketitle

\begin{abstract}
We prove that a balanced Boolean function on \(S_n\) whose Fourier transform is highly concentrated on the first two irreducible representations of \(S_n\), is close in structure to a dictatorship, a function which is determined by the image or pre-image of a single element. As a corollary, we obtain a stability result concerning extremal isoperimetric sets in the Cayley graph on $S_n$ generated by the transpositions.

Our proof works in the case where the expectation of the function is bounded away from 0 and 1. In contrast, \cite{EFF1} deals with Boolean functions of expectation $O(1/n)$ whose Fourier transform is highly concentrated on the first two irreducible representations of $S_n$. These need not be close to dictatorships; rather, they must be close to a union of a constant number of cosets of point-stabilizers.\\

\textbf{Keywords:} Fourier transform, stability, symmetric group.
\end{abstract}

\section{Introduction}
\subsection{Background}\label{sec:background}
This paper (together with \cite{EFF1} and \cite{EFF3}) is part of a trilogy dealing with stability and `quasi-stability' results concerning Boolean functions on the symmetric group, which are of `low complexity', in a Fourier-theoretic sense.

Let us begin with some notation and definitions that will enable us to present the Fourier-theoretic context of our results. Following this, the paper will be essentially Fourier-free, since Lemmas \ref{lem:TinU} and \ref{lem:TspansU} will translate the relevant Fourier notion into a more combinatorial one.

We write $[n] := \{1,2,\ldots,n\}$, and we let $S_n$ denote the symmetric group on $[n]$. If $i,j \in [n]$, we write $T_{ij} := \{\pi \in S_n:\ \pi(i)=j \}$. We call the $T_{ij}$'s the {\em 1-cosets} of $S_n$, since they are cosets of point-stabilizers. Similarly, for $t>1$, and for two ordered $t$-tuples of distinct elements of $[n]$, $I=(i_1,\ldots, i_t)$ and $J=(j_1,\ldots, j_t)$, we write 
$T_{IJ}:= \{\pi \in S_n:\ \pi(I)=J\}$. We call the $T_{IJ}$'s the {\em $t$-cosets} of $S_n$. If $\mathcal{F} \subset S_n$, we write $\chi_{\mathcal{F}}$ for its characteristic function, i.e.\ the Boolean function on $S_n$ with $\chi_{\mathcal{F}}(\pi)=1$ iff $\pi \in \cF$. Abusing notation slightly, we will often use $T_{ij}$ and $T_{IJ}$ to denote their own characteristic functions.

We say that a Boolean function $f\colon S_n \rightarrow \{0,1\}$ is a {\em dictatorship} if there exists $i \in [n]$ and $X \subset [n]$ such that $f(\pi) = \chi_{\{\pi(i) \in X\}}$ for all $\pi \in S_n$, or $f(\pi) = \chi_{\{\pi^{-1}(i) \in X\}}$ for all $\pi \in S_n$, i.e.\ iff $f$ is determined by the image or the preimage of just one element of $[n]$. It is easy to see that a Boolean function $f$ on $S_n$ is a dictatorship if and only if it is the characteristic function of a disjoint union of $1$-cosets, i.e.\ a sum of $T_{i,j}$'s.

If $f:S_n \to \mathbb{R}$, the {\em Fourier transform} of $f$ at an irreducible representation $\rho$ of $S_n$ is defined by
$$\hat{f}(\rho) = \frac{1}{n!}\sum_{\pi \in S_n} f(\pi)\rho(\pi).$$
Recall that the equivalence classes of irreducible representations of $S_n$ are indexed by partitions of $n$. We refer the reader to \cite{sagan} for background on the representation theory of $S_n$, and to \cite{terras} for background on the Fourier transform on non-Abelian groups.

For any non-negative integer $t$, we let $U_t$ denote the vector space of real-valued functions on $S_n$ whose Fourier transform is supported on irreducible representations indexed by partitions of $n$, whose largest part has size at least $n-t$. (Note that a partition $\lambda$ of $n$ has largest part of size at least $n-t$ if and only if $ \lambda \succeq (n-t,1^t)$, where $\succeq$ denotes the lexicographical order on partitions of $n$.) If $f$ is a real-valued function on $S_n$, we define the {\em degree} of $f$ to be the minimum $t$ such that $f \in U_t$. This is a measure of the complexity of $f$, analogous to the degree of a Boolean function on $\{0,1\}^n$. Indeed, it is precisely the minimum possible total degree of a polynomial in the $T_{ij}$'s which is equal to $f$.

Note that $U_0$ is the space of functions whose Fourier transform is supported on the trivial representation --- i.e., the space of constant functions. The space $U_1$, which is the main subject of this paper, is the space of functions whose Fourier transform is supported on the two irreducible constituents of the permutation representation. As promised, we now de-Fourierize this definition.

First, an easy fact:
\begin{lemma}\label{lem:TinU}
For any $t \in \mathbb{N}$, the indicators of $t$-cosets (the functions $T_{IJ}$), are in $U_t$.
\end{lemma}    

Next, a slightly more intricate fact, observed and proved in \cite{EFP}:
\begin{lemma}\label{lem:TspansU}
For any $t \in \mathbb{N}$, the $T_{IJ}$'s span $U_t$.
\end{lemma}   
\noindent In this paper, which deals only with $U_1$, we will henceforth use the definition $U_1 = \Span\{T_{ij}:\ i,j \in [n]\}$.

Finally, we recall a theorem from \cite{EFP}, which characterizes the Boolean functions in $U_1$.
\begin{theorem}[Ellis, Friedgut, Pilpel]
\label{thm:char}
Let $f: S_n \rightarrow \{0,1\}$ be in $U_1$. Then $f$ is a dictatorship. (Equivalently, $f$ is the characteristic function of a disjoint union of $1$-cosets.)
\end{theorem} 

The goal of the current paper, together with \cite{EFF1}, is to provide stability versions of Theorem \ref{thm:char}. This is in the spirit of similar projects in the Abelian case, which have proved extremely useful and applicable, see e.g.\ \cite{Bourgain}, \cite{FKN}, \cite{FriedgutJunta}, \cite{KindlerODonnell}, \cite{KindlerSafra} and \cite{NisanSzegedy}. The general idea in applications is to prove results in extremal combinatorics using Fourier analysis, and then use the Fourier stability results in order to deduce combinatorial stability results. A good example of this is Theorem \ref{thm:stableiso} in this paper, where we characterize the almost-extremal sets for the edge-isoperimetric inequality in the transposition graph on $S_n$ (the Cayley graph on $S_n$ generated by the transpositions). See \cite{EFF1} for more about applications in the symmetric group setting.

We remark that if $t \geq 2$, then a Boolean function in $U_t$ is not necessarily the characteristic function of a union of $t$-cosets. Theorem 27 in \cite{EFP} states that a Boolean function in $U_t$ is the characteristic function of a disjoint union of $t$-cosets, but this is false; a counterexample, and the error in the proof, is pointed out by the second author in \cite{F-note}. A counterexample when $t=2$ is as follows. Let $n \geq 8$. For any permutation $\pi \in S_n$, define $x = x(\pi) \in \{0,1\}^4$ by $x_i = \chi_{\{\pi(i) \in [4]\}}$, and consider the function
$$f:S_n \to \{0,1\}; \quad \pi \mapsto \chi_{\{x_1 \geq x_2 \geq x_3 \geq x_4 \text{ or } x_1 \leq x_2 \leq x_3 \leq x_4\}}.$$
It can be checked that $f \in U_2$, but the value of $f$ clearly cannot be determined by fixing the images of at most two elements of $[n]$, so neither $f$ nor $1-f$ is a union of 2-cosets. It is easy to use $f$ to construct a counterexample for each $t \geq 3$, by considering a product of $f$ with the characteristic function of the pointwise stabilizer of a $(t-2)$-set. We note that the main application of Theorem 27 in \cite{EFP} was to characterize (for large $n$) the $t$-intersecting families in $S_n$ of maximum size (i.e., to characterize the cases of equality in the Deza-Frankl conjecture); fortunately, this characterization follows immediately e.g.\ from the Hilton-Milner type result of the first author in \cite{tstability}, where the proof does not depend on Theorem 27 in \cite{EFP} (and indeed predates the latter).

The division between the three papers in our trilogy is as follows: in \cite{EFF1}, we deal with Boolean functions which are close to $U_1$, and have expectation $O(1/n)$. We prove that such a functions must be close to a sum of dictatorships --- equivalently, close to the characteristic function of a union of 1-cosets. In the current paper, we prove that Boolean functions that are close to $U_1$, and whose expectation is bounded away from 0 and 1, must be close to a single dictatorship. Finally, in \cite{EFF3}, we deal with Boolean functions close to $U_t$, with expectation $O(n^{-t})$; we prove that they must be close to the characteristic function of a union of $t$-cosets. (The latter is perhaps especially interesting, in view of the fact that a Boolean function in $U_t$ with expectation $O(n^{-t})$ is not necessarily the characteristic function of a union of $t$-cosets.) The term `quasi-stability' in the titles of the other two papers refers to the fact that a Boolean function such as $T_{11}+T_{22}-T_{11}\cdot T_{22}$ is 
$O(1/n^2)$ close to $U_1$, and is indeed $O(1/n^2)$ close to $T_{11}+T_{22}$, which is a sum of two dictatorships, but is not $O(1/n^2)$ close to any single dictatorship. In the case studied in this paper, we have {\em bona fide} stability, as the functions in question turn out to be close to a dictator. For both ranges of expectation we have studied, however, the Boolean functions which are close to $U_1$ are close to the characteristic function of a union of 1-cosets. Interestingly, the proof in this paper is quite different from the proof in \cite{EFF1}. It would be interesting to find a common proof that covers the complete spectrum of possible values of $\mathbb{E}[f]$.

The outline of the paper is as follows. In the rest of this section, we set up our notation, state our main result, and outline the proof. In section  \ref{sec2}, we prove the main theorem for the case of functions with expectation $=1/2$. Next, in section \ref{sec:general-case}, we adapt the proof to deal with functions with expectation bounded away from 0 and 1. Finally, in section \ref{sec4}, we give an application of our main theorem: a characterization of the almost-extremal sets for the edge-isoperimetric inequality for the transposition graph on $S_n$.

\subsection{Notation} \label{sec:definitions}
We now outline our notation systematically. If $\cF \subset S_n$, we let $\charf{\cF}$ denote the {\em characteristic function} of $\cF$, i.e.\ $\charf{\cF}:S_n\to\{0,1\}$ with $\charf{\cF}(\pi)=1$ iff $\pi \in \cF$. If $B$ is a statement, its {\em indicator} $\cond{B}$ is equal to~$1$ if $B$ is true and~$0$ if $B$ is false.

We write $[n]:=\{1,2,\ldots,n\}$. Let $S_n$ denote the symmetric group on $[n]$, i.e.\ the group of all permutations of $[n]$. For each $i,j \in [n]$, we define
\[ T_{ij} = \{ \pi \in S_n : \pi(i) = j \} \]
to be the set of all permutations sending $i$ to $j$; we call these the {\em 1-cosets} of $S_n$, as they are the cosets of stabilisers of points. Abusing notation slightly, we will also use $T_{ij}$ to denote its own characteristic vector, more properly written as $\charf{T_{ij}}$.

We define $U_1(n)$ to be the subspace of $\mathbb{R}^{S_n}$ spanned by $\{T_{ij} : i,j \in [n] \}$. When $n$ is understood, we abbreviate this to $U_1$.

We equip \(\mathbb{R}^{S_n}\) with the inner product induced by the uniform probability measure on \(S_n\):
\[\langle f,g \rangle = \frac{1}{n!}\sum_{\pi \in S_n} f(\pi) g(\pi).\]
The {\em expectation} of a real-valued function on $S_n$ will mean the expectation with respect to the uniform probability measure, i.e.
$$\mathbb{E}[f] = \frac{1}{n!}\sum_{\pi \in S_n} f(\pi).$$
We let $|| \cdot ||_2$ denote the induced Euclidean norm; i.e.
\[||f||_2 = \sqrt{\mathbb{E}[f^2]} = \sqrt{\frac{1}{n!}\sum_{\pi \in S_n} f(\pi)^2}.\]
The distance between functions, or between a function and a subspace, will mean the Euclidean distance as defined by this norm.

Throughout, if \(u\) and \(v\) are functions of several variables, the notation \(u = O(v)\) will mean that there exists an absolute constant \(C\) (not depending upon any of the variables) such that \(|u| \leq C|v|\) pointwise.

The notation $\upto{x}{\epsilon}$ is shorthand for the closed interval $[x-\epsilon,x+\epsilon]$. If $y \in \upto{x}{\epsilon}$, then we say that {\em $y$ is $\epsilon$-close to $x$}. If $y \notin \upto{x}{\epsilon}$, then we say that {\em $y$ is $\epsilon$-far from $x$}. For a set $S$, we say that {\em $x$ is $\epsilon$-close to $S$} if $|x-y| \leq \epsilon$ for some $y \in S$. Otherwise, we say that {\em $x$ is $\epsilon$-far from $S$}.

We will be dealing throughout with functions on finite probability spaces (i.e., with random variables); we will frequently refer to these simply as `functions' (rather than as `random variables'), when the underlying probability space is understood.
\newpage

\subsection{Main result} \label{sec:main}
Our main goal in this paper is to prove the following theorem.
\begin{restatable}{theorem}{maintheorem} \label{thm:main}
 Let $\cF \subset S_n$ be a family of permutations with size $|\cF| = c\cdot n!$, satisfying
 \[ \EE[(f-f_1)^2] = \error, \]
where $f = 2\chi_\cF-1$, and $f_1$ is the orthogonal projection of $f$ onto $U_1$. Then there exists a family $\cG \subset S_n$ which is a union of $dn$ disjoint 1-cosets, such that
\[
|d - c| = O\left(\error^{1/7} + \frac{1}{n^{1/3}}\right)
\quad \text{and} \quad
\frac{|\cG \triangle \cF|}{n!} = O\left(\frac{1}{\eta} \left( \error^{1/7} + \frac{1}{n^{1/3}}\right)\right),
\]
where $\eta = \min\{c,1-c\}$.
\end{restatable}

\begin{remark}
It is convenient to work with the $\pm 1$-valued function $f$, rather than the $0/1$-valued function $\chi_{\mathcal{F}}$. Note that if $(\chi_{\mathcal{F}})_1$ denotes the orthogonal projection of $\chi_{\mathcal{F}}$ onto $U_1$, then the square of the Euclidean distance of $\chi_{\mathcal{F}}$ from $U_1$ is
\[\EE[(\chi_\mathcal{F}-(\chi_{\mathcal{F}})_1)^2] = \tfrac{1}{4}\EE[(f-f_1)^2].\]
Throughout the proof, a {\em Boolean function} will mean a function taking values in $\{\pm 1\}$, rather than $\{0,1\}$.
\end{remark}


For the entire proof, we will make the assumptions
\begin{gather}
\label{eq:n-assumption}
n \geq 4, \\
\label{eq:e-assumption}
\frac{1}{n^{7/3}} \leq \error < \errub,
\end{gather}
where $\errub > 0$ depends only upon $c$. Later, we will show how to get rid of these assumptions. During the proof, we will use the phrase \emph{since $\error$ is small enough, $P$ holds} to mean that for some $\errub >0$, the statement $P$ follows from $\error < \errub$.

For pedagogical reasons, we will first assume that $c = 1/2$. This assumption does not affect the proof very much, but it simplifies the expressions appearing therein. After completing the proof in this case, we will show how to extend it to general~$c$, carefully noting the relation between $\errub$ and $c$.

\subsection{Proof overview} \label{sec:overview}
We adopt a simple canonical way to express $f_1$ as a linear combination of the form
\[ f_1 = \sum_{i,j} a_{ij} T_{ij}; \]
we then study the matrix of coefficients $(a_{ij})$. This offers a nice visualization of the function, due to the observation that 
 \[ f_1(\pi) = \sum_{i=1}^n a_{i\pi(i)}, \] 
 i.e. $f_1(\pi)$ is equal to the sum of the entries on a {\em generalised diagonal} of the matrix $(a_{ij})$. (A generalised diagonal of an $n \times n$ matrix is a set of $n$ entries with one entry from each row and one from each column, so corresponds to a permutation of $\{1,2,\ldots,n\}$.)
 
Note that the $T_{ij}$'s are linearly dependent (the dimension of $U_1$ is only $(n-1)^2+1$, whereas there are $n^2$ different $T_{ij}$'s), so there are many possible ways to represent $f_1$ in such a manner. It turns out that a particularly useful choice (when $c=1/2$) is
\begin{equation} \tag{\ref{eq:a-formula}}
 a_{ij} = (n-1) \langle f, T_{ij} \rangle.
\end{equation}

\newpage
To illustrate this, here is the matrix corresponding to the dictatorship $\cF = \{\pi \in S_n : 1 \leq \pi(1) \leq n/2 \}$ (where $n$ is even):
\canonicalexample

This matrix exemplifies the usefulness of our choice $  a_{ij} = (n-1) \langle f, T_{ij} \rangle $: the entries that are significant for our dictatorship (row 1, which depends on the image of 1) are all close in absolute value to 1, whereas all other entries are close to 0.
The idea of the proof is to discover some properties of the matrix $(a_{ij})$, and then show that they imply
that it looks roughly like the matrix above --- namely, that it has precisely one row or column in which almost half the entries are very close to $1$ and almost half the entries are very close to $-1$, and that almost all the other entries in the matrix are very close to $0$.

The proof breaks down into two main parts. In the first part, we show that for almost all $\pi \in S_n$, the generalised diagonal defined by $\pi$, namely $\{a_{i\pi(i)} : 1 \leq i \leq n\}$, has precisely one entry which is `large' (close to 1 or $-1$), and all the rest of its entries are small. In the second part, we deduce that~$(a_{ij})$ must have either a row or a column, almost all of whose entries are large. This will enable us to complete the proof.

\subsubsection*{Part 1}
\textbf {Step 1 (\autoref{sec:restrictions}, \autoref{sec:decomposition}).} Consider any two sets $X,Y \subset \{1,\ldots,n\}$ with $|X|=|Y|$, and the corresponding set of permutations
\[ T_{X,Y} = \{ \pi \in S_n : \pi(X) = Y \}. \]
When calculating $f_1$ on $T_{X,Y}$, we only need to look at the submatrices of~$(a_{ij})$ defined by $X\times Y$ and by $\overline{X}\times \overline{Y}$. So it is natural to define, for $\pi_1$ a bijection from $X$ to $Y$ and $\pi_2$ a bijection from $\overline{X}$ to $\overline{Y}$,
\[
 g_1(\pi_1) = \sum_{i \in X} a_{i\pi_1(i)}, \quad
 g_2(\pi_2) = \sum_{i \in \overline{X}} a_{i \pi_2(i)}, \quad
 g(\pi_1,\pi_2) = g_1(\pi_1) + g_2(\pi_2),
\]
where $(\pi_1,\pi_2)$ denotes the permutation of $S_n$ whose restrictions to $X$ and $\overline{X}$ are $\pi_1$ and $\pi_2$ respectively. Notice that $g$ is just the restriction of $f_1$ to $T_{X,Y}$. For most choices of $X,Y$, these functions are `well-behaved', meaning that all of the following hold (\autoref{sec:restrictions}):
\begin{itemize}
 \item For most permutations $\pi \in T_{X,Y}$, $g(\pi)$ is close to $\pm 1$.
 \item Furthermore, the function $g$ is close to $\pm 1$ in an $L^2$ sense.
 \item Both $\EE g_1$ and $\EE g_2$ are close to their expected value, $0$.
\end{itemize}
Next, we note that, crucially, the function $g$ (which can be viewed as a random variable) is the sum of two independent random variables $g_1,g_2$, and yet is concentrated near $1$ and $-1$. How can that happen? We show (\autoref{sec:decomposition}) that it must be the case that one of the $g_i$'s (say $g_1$) is concentrated around a constant $C$, and that the other (say $g_2$) is concentrated around two values, $-C-1$ and $-C+1$. Using the observations above, it follows that $C$ is very close to $0$.\\
\\
\textbf{Step 2 (\autoref{sec:partition}).} For any permutation $\pi \in S_n$, we consider all pairs $(X,Y)$ compatible with it, i.e. all pairs $(X,\pi(X))$. For most choices of $\pi$ and for most choices of compatible $(X,Y)$, it will be true that one of $g_1(\pi_1),g_2(\pi_2)$ is close to $0$, and the other is close to $\pm 1$. Note that
\[ g_1(\pi_1) = \sum_{i \in X} a_{i\pi(i)}, \quad g_2(\pi_2) = \sum_{i \in \overline{X}} a_{i\pi(i)}. \]
Put differently, for most permutations $\pi \in S_n$ it is true that for most ways of splitting the generalised diagonal 
 $D = \{ a_{i\pi(i)} : 1 \leq i \leq n \}$ into two parts, one part sums to roughly $0$, and the other to roughly $\pm 1$. That can only happen if almost all entries in $D$ are small, and one is close in magnitude to $1$. 

\subsubsection*{Part 2 (\autoref{sec:strong})}
This part uses induction on $n$ to prove the following claim. If an $n \times n$ matrix satisfies \emph{property $Q(\delta)$}, namely a $(1-\delta)$-fraction of its generalised diagonals have a single entry which is large in magnitude, then the matrix has a \emph{strong line} --- either a row or a column, a $(1-C\delta)$-fraction of whose entries are large. (Here, $C$ does not depend on $n$.)

\textbf{Base case.} When $n$ is small compared to $1/\delta$, we can prove directly that there is a line where {\em all} of the entries are large. 

\textbf{Induction step.} Given an $n\times n$ matrix $M$ satisfying $Q(\delta)$ and a set $X$ of $n/2$ rows, we can always find a set $Y$ of $n/2$ columns such that either $X\times Y$ or $\overline{X}\times \overline{Y}$ also satisfy $Q(\delta)$. The induction hypothesis shows that the relevant submatrix has a strong line. The strong lines for different choices of $X$ must be the same (on the same row or column of $M$), since otherwise the probability that a generalized diagonal passes through two large entries would be too big. Altogether, these strong lines constitute a line~$\ell$ which is almost as strong as required. A small bootstrapping argument shows that $\ell$ must indeed have the required number of large entries.

\subsubsection*{Culmination (\autoref{sec:culmination})}
At this stage of the proof, we know that the matrix ~$(a_{ij})$ has a line, say row $i$, almost all of whose entries are close either to $-1$ or to $1$. It follows that for most $j$, it holds that $(n-1)! \langle f, T_{ij} \rangle$ is close to~$0$ or to~$1$. The disjoint union of the 1-cosets corresponding to those entries close to~$1$ form a good approximation to~$\cF$.\\

Part 2 is largely independent of Part 1. Part 1 shows that most generalized diagonals of the matrix $(a_{ij})$ are composed of one large entry and $n-1$ small entries. Part 2 abstracts this situation, and deduces the existence of a strong line. The results of Part 2 work for any definition of which entries are large and which are small, and so are of independent interest.

\paragraph*{Glossary of terminology} \emph{Restrictions} are defined in the beginning of~\autoref{sec:restrictions}. \emph{Typical restrictions} are defined in the end of~\autoref{sec:restrictions}. \emph{Good restrictions} are defined in the beginning of~\autoref{sec:strong}. Functions which are \emph{almost Boolean} or \emph{almost close to~$C$} are defined in~\autoref{sec:restrictions}, just before \autoref{lem:typical-restriction}. \emph{Partitions}, \emph{good partitions} and \emph{good permutations} are defined in the beginning of~\autoref{sec:partition}. \emph{Small} and \emph{large} entries are defined in the beginning of~\autoref{sec:strong}. \emph{Strong lines} (as well as strong rows and columns) are defined in~\autoref{sec:strong}, just after \autoref{lem:good-half}.

\section{Proof when \texorpdfstring{$c=1/2$}{c=1/2}}\label{sec2}
\subsection{Matrix representation} \label{sec:matrix}
Let $\mathcal{F}$ and $f$ be as in the statement of the theorem. Since $f_1 \in U_1$, it can be represented as a linear combination of 1-cosets $T_{ij}$. We single out one such representation:
\begin{equation} \label{eq:a-formula}
 a_{ij} = (n-1) \langle f, T_{ij} \rangle.
\end{equation}

We start by showing that the $a_{ij}$ do indeed represent $f_1$.
\begin{lemma} \label{lem:a}
 We have
 \[ f_1 = \sum_{i,j} a_{ij} T_{ij}. \]
 Furthermore, each row and each column of the matrix $(a_{ij})$ sums to zero:
 \[ \sum_j a_{ij} = 0 \quad \forall i \in [n],\quad \sum_i a_{ij} = 0\quad \forall j \in [n].\]
 For each permutation $\pi \in S_n$, we have
 \[ f_1(\pi) = \sum_i a_{i\pi(i)}. \]
\end{lemma}
\begin{proof}
The `real' proof of this fact uses the Fourier inversion formula, and the characters of the first two irreducible representations of $S_n$. This is how we derived the formula. However, to avoid digression, we offer a simpler, {\em ad hoc} argument.

The second statement follows from a simple calculation. For each $i$,
\[
 \sum_j \langle f,T_{ij} \rangle = \langle f,1 \rangle = 2 \langle \charf{\cF},1 \rangle - \langle 1,1 \rangle = 0,
\]
so
$$\sum_j a_{ij}  = 0.$$
Similarly, for each $j$,
\[
 \sum_i \langle f,T_{ij} \rangle = \langle f,1 \rangle = 0,
\]
so
$$\sum_i a_{ij}  = 0.$$

For the first statement, both sides of the equation are in $U_1$, so it is enough to show that both sides have the same inner product with each $T_{ij}$. Note that $\langle T_{ij},T_{ij} \rangle = 1/n$, $\langle T_{ij},T_{kl} \rangle = \tfrac{1}{n(n-1)}$ if $i \neq k$ and $j \neq l$, and $\langle T_{ij},T_{kl} \rangle = 0$ if $i \neq k$ or $j \neq l$. Therefore,
 \[
  \big\langle \sum_{k,l} a_{kl} T_{kl}, T_{ij} \big\rangle =
  \frac{a_{ij}}{n} + \frac{1}{n(n-1)} \sum_{\substack{k \neq i\\ l \neq j}} a_{kl} = 
  \frac{a_{ij}}{n} + \frac{a_{ij}}{n(n-1)} = \frac{a_{ij}}{n-1} = \langle f, T_{ij} \rangle = \langle f_1, T_{ij} \rangle,
 \]

using
\begin{equation}\label{eq:zerosum} \sum_{\substack{k \neq i\\ l \neq j}} a_{kl} = \sum_{k,l}a_{kl} - \sum_{l}a_{il}-\sum_{k} a_{kj}+a_{ij} = a_{ij}.\end{equation}

Finally, the formula for $f_1(\pi)$ follows immediately from the first statement.
\end{proof}

The preceding lemma shows that each value of $f_1$ is equal to the sum of a generalised diagonal in the matrix $(a_{ij})$.

Next, we calculate the $\ell^2$ norm of the vector formed by the entries $a_{ij}$.
\begin{lemma} \label{lem:a-l2}
 We have
 \[ \sum_{i,j} a_{ij}^2 = (n-1)(1-\error). \]
\end{lemma}
\begin{proof}
 Since $f_1$ is an orthogonal projection of $f$, we have
 \[
  1 = \|f\|_2^2 = \|f_1\|_2^2 + \|f-f_1\|_2^2 = \|f_1\|_2^2 + \error.
 \]
 Therefore, $\|f_1\|_2^2 = 1 - \error$. On the other hand, we have
 \begin{align*}
 \|f_1\|_2^2 &= \sum_{i,j,k,l} a_{ij} a_{kl} \langle T_{ij},T_{kl} \rangle \\ &=
 \frac{1}{n} \sum_{i,j} a_{ij}^2 + \frac{1}{n(n-1)} \sum_{i,j} \sum_{\substack{k \neq i\\ l \neq j}} a_{ij}a_{kl} \\ &=
 \frac{1}{n} \sum_{i,j} a_{ij}^2 + \frac{1}{n(n-1)} \sum_{i,j} a_{ij}^2 = \frac{1}{n-1} \sum_{i,j} a_{ij}^2,
 \end{align*}
using \ref{eq:zerosum}.
\end{proof}

\subsection{Random restrictions} \label{sec:restrictions}
For $X,Y \subset [n]$ of equal size, let $T_{X,Y}$ denote the set of all permutations sending $X$ to $Y$:
\[ T_{X,Y} = \{ \pi \in S_n : \pi(X) = Y) \}. \]
We call such a pair $(X,Y)$ a \emph{restriction}. Let $g(X,Y)$ denote the subvector of $f_1$ supported on $T_{X,Y}$. 

The final part of \autoref{lem:a} shows that every value of $f_1$ is the sum of a generalized diagonal of $(a_{ij})$. It is natural to decompose $g(X,Y)$ into two functions, one depending on the submatrix supported by $X \times Y$, the other depending on the submatrix supported by $\overline{X} \times \overline{Y}$:
\[ g_1(X,Y) = \sum_{i \in X, j \in Y} a_{ij} T_{ij}, \quad g_2(X,Y) = g_1(\overline{X},\overline{Y}). \]
(In the definition of $g_1$, $T_{ij}$ is, strictly speaking, the restriction of $\chi_{T_{ij}}$ to $T_{X,Y}$.) \autoref{lem:a} immediately implies that $g(X,Y) = g_1(X,Y) + g_2(X,Y)$. Note that $g_1(X,Y)$ and $g_2(X,Y)$ are both supported on $T_{X,Y}$.

We now define a probability distribution $\rand$ over the set of all restrictions, as follows. Each $i \in [n]$ is included in $X$ independently at random with probability $1/2$. Then, $Y$ is chosen uniformly at random from all sets of size $|X|$. Note that this definition is symmetric between $X$ and $Y$, and furthermore, $(X,Y)$ has the same distribution as $(\overline{X},\overline{Y})$.

Most of this subsection will be devoted to the study of properties of typical restrictions. We start by calculating the mean and variance of $\EE[g_1(X,Y)]$ when $(X,Y) \sim \rand$. This will enable us to show that $\EE[g_1(X,Y)]$ and $\EE[g_2(X,Y)]$ are typically small in magnitude.
\begin{lemma} \label{lem:e-moments}
 Let $X,Y \subset [n]$ with $|X|=|Y|$, and define 
 $$m(X,Y) = \EE_{T_{X,Y}}[g_1(X,Y)],$$
 where the expectation is with respect to the uniform probability measure on $T_{X,Y}$.
 
 If $(X,Y) \sim \mathcal{R}$, then the mean and variance of $m(X,Y)$ with respect to $\mathcal{R}$ satisfy
 $$\mathbb{E}_{\mathcal{R}} [m] = 0$$
  and
  $$\mathbb{V}_{\mathcal{R}}[m] \leq \tfrac{1}{2n}.$$
\end{lemma}
\begin{proof}
 We start with a formula for $m(X,Y)$:
\[
 m(X,Y) = \sum_{i\in X, j\in Y} a_{ij} \EE_{T_{X,Y}}[T_{ij}] = \frac{1}{|X|} \sum_{i\in X, j\in Y} a_{ij}.
\]
 Conditioned upon $|X|$, we have
\[
 \mathbb{E}_{\mathcal{R}}[m(X,Y)| |X|] = \frac{1}{|X|} \frac{|X|^2}{n^2} \sum_{i,j} a_{ij} = 0.
\]
 Hence, $\mathbb{E}_{\mathcal{R}}[m] = 0$.

 The next step is to calculate $\mathbb{E}_{\mathcal{R}}[m^2]$. Expanding the formula, we get
 \begin{align*}
  |X|^2 m(X,Y)^2 & = \sum_{i,j} \cond{\{i \in X, j \in Y\}} a_{ij}^2 +
  \sum_i \sum_{j \neq l} \cond{\{i \in X, j,l \in Y\}} a_{ij} a_{il}\\
  & + \sum_j \sum_{i \neq k} \cond{\{i,k \in X, j \in Y\}} a_{ij} a_{kj} +
 \sum_{\substack{k \neq i,\\ l \neq j}} \cond{\{i,k \in X, j,l \in Y\}} a_{ij} a_{kl}.
 \end{align*}
 Taking expectations, we get
 \begin{align*}
  |X|^2 \mathbb{E}_{\mathcal{R}}[m(X,Y)^2 | |X|] &=
  \Pr[i \in X \land j \in Y] \sum_{i,j} a_{ij}^2 +
  \Pr[i \in X \land j,l \in Y] \sum_i \sum_{l \neq j} a_{ij} a_{il} \\ &+
  \Pr[i,k \in X \land j \in Y] \sum_{k \neq i} \sum_j a_{ij} a_{kj} +
  \Pr[i,k \in X \land j,l \in Y] \sum_{\substack{k \neq i,\\ l \neq j}} a_{ij} a_{kl}.
 \end{align*}
 Using \ref{eq:zerosum}, together with
 \[
  \sum_{l \neq j} a_{il} = \sum_{k \neq i} a_{kj} = -a_{ij}\]
 (from \autoref{lem:a}) we obtain
 \begin{align*}
  n^2 \mathbb{E}_{\mathcal{R}}[m(X,Y)^2 | |X|]& =
  \sum_{i,j} a_{ij}^2 -
  2\frac{|X|-1}{n-1} \sum_{i,j} a_{ij}^2 +
  \frac{(|X|-1)^2}{(n-1)^2} \sum_{i,j} a_{ij}^2 \\
  &=
  \left(1 - \frac{|X|-1}{n-1}\right)^2 \sum_{i,j} a_{ij}^2.
 \end{align*}
 Taking expectations and using the estimate $\sum_{i,j} a_{ij}^2 \leq n-1$ provided by \autoref{lem:a-l2}, we conclude that
 \begin{align*}
  \mathbb{E}_{\mathcal{R}}[m(X,Y)^2] & = \frac{n+1}{4n(n-1)^2} \sum_{i,j} a_{ij}^2\\
  & \leq \frac{n+1}{4n(n-1)} \\
  & \leq \frac{1}{2n}.
 \end{align*}
\end{proof}

The following lemma states some properties that a random restriction enjoys with probability close to 1. The lemma uses the following nomenclature for functions on a probability space (a.k.a. random variables):
\begin{itemize}
 \item A function $\phi$ is \emph{$(\delta,\epsilon)$-almost Boolean} if with probability at least $1-\delta$, $\phi$ is $\epsilon$-close to~$\pm 1$. In symbols, \[ \Pr[|\phi| \in \upto{1}{\epsilon}] \geq 1 - \delta. \]
 \item A function $\phi$ is \emph{$(\delta,\epsilon)$-almost close to~$C$} if with probability at least $1-\delta$, $\phi$ is $\epsilon$-close to~$C$. In symbols, \[ \Pr[|\phi - C| \leq \epsilon] \geq 1 - \delta. \]
\end{itemize}

\begin{lemma} \label{lem:typical-restriction}
 Let $(X,Y) \sim \rand$. With probability at least $1-3\error^{1/7}$, $(X,Y)$ satisfies the following properties:
\begin{enumerate}[(a)]
 \item \label{lem:typical-restriction:bool} $g(X,Y)$ is $(\error^{4/7},\error^{1/7})$-almost Boolean.
 \item \label{lem:typical-restriction:exp} $\EE[g_1(X,Y)]$ and $\EE[g_2(X,Y)]$ are $\error^{1/7}$-close to zero.
 \item \label{lem:typical-restriction:l2} $\EE[(|g(X,Y)|-1)^2] \leq \error^{6/7}$.
\end{enumerate}
\end{lemma}
\begin{proof}
 We claim that each of the different parts holds with probability at least $1-\error^{1/7}$. The lemma follows using a union bound. We will use the fact that selecting a random partition $(X,Y) \sim \mathcal{R}$ and then selecting a uniform random element in $T_{X,Y}$ is the same as choosing a uniform random permutation. This holds because for any $X$, the sets $(T_{X,Y}:\ |Y|=|X|)$ partition $S_n$.

 We first deal with part (a). Suppose for a contradiction that
 $$\Pr_{\mathcal{R}} \left[\Pr_{T_{X,Y}}[|g(X,Y)| \notin \upto{1}{\error^{1/7}}] \geq \error^{4/7}\right] > \error^{1/7}.$$
 It follows that $\Pr[|f_1| \notin \upto{1}{\error^{1/7}}] > \error^{5/7}$. This implies that
\[ \EE[(f_1-f)^2] \geq \EE[(|f_1|-1)^2] > \error^{5/7} (\error^{1/7})^2 = \error. \]
 This contradicts $\EE[(f_1-f)^2]=\error$, proving the claim for part (a).

 For part (b), we use \autoref{lem:e-moments}, which gives the mean and variance (with respect to $\mathcal{R}$) of $\EE_{T_{X,Y}}[g_1(X,Y)]$. We have
\[ \Pr_{\mathcal{R}}\left[\left|\EE_{T_{X,Y}}[g_1(X,Y)]\right| \geq \frac{1}{\error^{1/14}\sqrt{n}}\right] \leq \frac{\error^{1/7}n}{2n} = \frac{1}{2} \error^{1/7}, \]
using Chebyshev's inequality. Assumption~\ref{eq:e-assumption} states that $\error \geq 1/n^{7/3}$. Hence $\error^{3/14} \geq 1/\sqrt{n}$, and so $\error^{1/7} \geq \error^{-1/14}/\sqrt{n}$.
 Therefore,
 \[ \Pr_{\mathcal{R}}\left[\left|\EE_{T_{X,Y}}[g_1(X,Y)]\right| \geq \error^{1/7}\right] \leq \frac{1}{2} \error^{1/7}.\]
 Since $(\overline{X},\overline{Y}) \sim \rand$, the same holds for $g_2(X,Y)$. The claim for part (b) follows, using a union bound.

 For part (c), the starting point is
\[ \EE[(|f_1|-1)^2] \leq \EE[(f_1-f)^2] = \error. \]
 Therefore,
$$ \EE_{\mathcal{R}}\left[\EE_{T_{X,Y}}[(|g(X,Y)|-1)^2]\right] = \EE[(|f_1|-1)^2] \leq \error.$$
 The claim now follows from Markov's inequality.
\end{proof}

We call a restriction {\em typical} if it satisfies the properties (a), (b) and (c) in \autoref{lem:typical-restriction}.

\subsection{Decomposition under a typical restriction} \label{sec:decomposition}

In this subsection, we show that if $(X,Y)$ is a typical restriction (meaning a restriction satisfying the properties listed in \autoref{lem:typical-restriction}), then the functions $g_1(X,Y)$ and $g_2(X,Y)$ have a particularly simple structure: up to translation, one of them is almost constant, and the other is almost Boolean. 

Many of the lemmas in this subsection start by assuming that a particular restriction $(X,Y)$ is typical. In these lemmas, we will write $g,g_1,g_2$ for $g(X,Y),g_1(X,Y),g_2(X,Y)$.

The following technical lemma tackles the following situation. Suppose that some function $\phi$ is almost close to~$C_0$. Can we deduce that $C_0 \approx \EE \phi$? The lemma gives a sufficient condition (in the case $C_0=0$).
\begin{lemma} \label{lem:mean-mode}
 Suppose that a function $\phi$ on a probability space satisfies the following properties:
\begin{enumerate}[(a)]
 \item The function $\phi$ is $(p,\epsilon)$-almost close to 0.
 \item There exists $C \in \mathbb{R}$ such that $\EE[(|\phi + C|-1)^2] \leq \delta$.
\end{enumerate}
 Then 
 \[
  |\EE[\phi]| \leq 3\epsilon + 3p + 6\sqrt{\frac{\delta}{1-p}}. \qedhere
 \]
\end{lemma}

\begin{remark}
Condition (b) says that $\phi+C$ is almost Boolean in the $L^2$ sense. As we shall see, conditions (a) and (b) together imply that $C$ must be close to 1 or close to $-1$.
\end{remark}

\begin{proof}
 Without loss of generality, we may assume that $C \geq 0$. We start by establishing the bound
 \begin{equation} \label{eq:mean-mode-c-bound}
  |1-C| \leq \epsilon + \sqrt{\frac{\delta}{1-p}}.  
 \end{equation}
 We distinguish between three cases: $C < 1-\epsilon$, $C > 1+\epsilon$ and $|1-C| \leq \epsilon$. In the latter case, we already have the desired bound.

 Suppose $C < 1-\epsilon$. Whenever $|\phi| \leq \epsilon$, we have
 \[ 1 - |\phi + C| \geq 1 - C - \epsilon > 0. \]
 Since this happens with probability at least $1-p$, we deduce that $(1-p)(1-C-\epsilon)^2 \leq \delta$, verifying~\eqref{eq:mean-mode-c-bound}.
 
 Suppose next that $C > 1 + \epsilon$. Whenever $|\phi| \leq \epsilon$, we have
 \[ |\phi + C| - 1 \geq C - \epsilon - 1 > 0. \]
 Since this happens with probability at least $1-p$, we deduce that $(1-p)(C-\epsilon-1)^2 \leq \delta$, again verifying~\eqref{eq:mean-mode-c-bound}.
 This completes the proof of~\eqref{eq:mean-mode-c-bound}.

 When $|t| \geq 1$, we have $|t| \leq t^2$, and so
 \[
  \EE[||\phi + C|-1|\charf{\{\phi+C \geq 2\}}] \leq \EE[(|\phi+C|-1)^2\charf{\{\phi+C \geq 2\}}] \leq \delta.
 \]
 The triangle inequality implies that
 \[
  \left| \EE[(\phi+C-1)\charf{\{\phi+C \geq 2\}}] \right| \leq
  \EE[|\phi+C-1|\charf{\{\phi+C \geq 2\}}] =
  \EE[||\phi+C|-1|\charf{\{\phi+C \geq 2\}}] \leq \delta.
 \]
 When $t \leq -2$, we have $|t-1| = 1-t \leq 3(-t-1) = 3 (|t|-1) = 3||t|-1|$, and so, as before,
 \[
  \EE[|\phi+C-1|\charf{\{\phi+C \leq -2\}}] \leq 3\EE[||\phi+C|-1|\charf{\{\phi+C \leq -2\}}] \leq 3\delta.
 \]
 The triangle inequality implies that
 \[
  \left| \EE[(\phi+C-1)\charf{\{\phi+C \leq -2\}}] \right| \leq 3\delta.  
 \]
 Combining the two together, we get
 \[
  \left| \EE[(\phi+C-1)\charf{\{|\phi+C| \geq 2\}}] \right| \leq 4\delta.  
 \]
 
 Define $\psi = \phi + C - 1$. Rewriting the last inequality in terms of $\psi$, we have
 \[
  \left| \EE[\psi\charf{\{|\psi+1| \geq 2\}}] \right| \leq 4\delta.  
 \]
 When $|\psi+1| \leq 2$, $|\psi| \leq 3$. When $|\phi| \leq \epsilon$, we have $|\psi| \leq |1-C| + \epsilon \leq 2\epsilon + \sqrt{\delta/(1-p)}$. Therefore
 \begin{align*}
  |\EE[\psi]| &\leq
  \left| \EE[\psi\charf{\{|\phi| \leq \epsilon\}}] \right| +
  \left| \EE[\psi\charf{\{|\phi| > \epsilon \text{ and }|\psi+1| \leq 2\}}] \right| +
  \left| \EE[\psi\charf{\{|\psi+1| \geq 2\}}] \right| \\ &\leq
  2\epsilon + \sqrt{\frac{\delta}{1-p}} + 3p + 4\delta. 
\end{align*}
 We conclude that
 \[
  |\EE[\phi]| \leq |\EE[\psi]| + |1-C| \leq 3\epsilon + 2\sqrt{\frac{\delta}{1-p}} + 3p + 4\delta. \qedhere
 \]
\end{proof}

Our first key step is the following lemma, which uses the fact that $g_1(X,Y)$ and $g_2(X,Y)$ are independent pieces of $g(X,Y)$ to deduce that, up to translation, both are close to Boolean. Moreover, at least one of them is close to being constant.

We will use the following notation, when a restriction $(X,Y)$ is understood. For a permutation $\pi$, $\pi_1 = \pi|_X$ denotes its restriction to $X$, and $\pi_2 = \pi|_{\overline{X}}$ denotes its restriction to $\overline{X}$. Thus, $g_1$ depends only upon $\pi_1$, and $g_2$ depends only upon $\pi_2$.

\begin{lemma} \label{lem:deltas}
 Suppose $(X,Y)$ is a typical restriction. Choose $\alpha,\beta$ uniformly at random from $T_{X,Y}$. Then with probability at least $1-8\error^{2/7}$, one of the following three cases holds:
 \begin{enumerate}[(a)]
  \item $|g_1(\alpha_1) - g_1(\beta_1)|$ and $|g_2(\alpha_2) - g_2(\beta_2)| \leq 2\error^{1/7}$.
  \item $|g_1(\alpha_1) - g_1(\beta_1)| \leq 2\error^{1/7}$ and $|g_2(\alpha_2) - g_2(\beta_2)| \in \upto{2}{2\error^{1/7}}$.
  \item $|g_1(\alpha_1) - g_1(\beta_1)| \in \upto{2}{2\error^{1/7}}$ and $|g_2(\alpha_2) - g_2(\beta_2)| \leq 2\error^{1/7}$.
 \end{enumerate}
\end{lemma}
\begin{proof}
Typicality implies that $\Pr[||g|-1| > \error^{1/7}] < \error^{4/7}$. This implies that with probability at least $1-\error^{2/7}$ over the choice of $\pi_1$, it is true that $\Pr_{\pi_2}[||g(\pi)|-1| > \error^{1/7}] < \error^{2/7}$. So with probability at least $1-2\error^{2/7}$ over the choice of $(\pi_1,\pi_2)$, it is true that $|g(\pi)|$ is $\error^{1/7}$-close to~$1$. Thus with probability at least $1-8\error^{2/7}$ over the choice of $\alpha,\beta$, all of the following are $\error^{1/7}$-close in magnitude to~$1$:
\begin{align*}
x &= g_1(\alpha_1) + g_2(\alpha_2), & y &= g_1(\alpha_1) + g_2(\beta_2), \\ z &= g_1(\beta_1) + g_2(\alpha_2), & w &= g_1(\beta_1) + g_2(\beta_2).
\end{align*}
Since $\error$ is small enough, each of these four values is unambiguously close to either~$1$ or~$-1$.

If $x,y,z$ are all close to the same value, then
\[ |g_1(\alpha_1) - g_1(\beta_1)|,|g_2(\alpha_2) - g_2(\beta_2)| \leq 2\error^{1/7}.\]

If $x$ and $y$ are close to different values, then
\[ |g_2(\alpha_2) - g_2(\beta_2)| \in \upto{2}{2\error^{1/7}}. \]
Without loss of generality, we may assume that $x$ is close to~$1$ and $y$ is close to~$-1$. Then $g_2(\alpha_2) - g_2(\beta_2) \in \upto{2}{2\error^{1/7}}$. If $z$ is close to~$-1$, then $g_1(\alpha_1) - g_1(\beta_1) \in \upto{2}{2\error^{1/7}}$. But this implies that
\[ w = x - (g_1(\alpha_1) - g_1(\beta_1)) - (g_2(\alpha_2) - g_2(\beta_2)) \in \upto{-3}{5\error^{1/7}}. \]
Since $\error$ is small enough, $-3 + 5\error^{1/7} < -1-\error^{1/7}$, and we reach a contradiction. So when $x$ and $y$ are close to different values, $x$ and $z$ must be close to the same value. This implies that
\[ |g_1(\alpha_1) - g_1(\beta_1)| \leq 2\error^{1/7}.\]

If $x$ and $z$ are close to different values, then we similarly obtain
\[
|g_1(\alpha_1) - g_1(\beta_1)| \in \upto{2}{2\error^{1/7}}, \quad
|g_2(\alpha_2) - g_2(\beta_2)| \leq 2\error^{1/7}.
\]
These cases are exhaustive.
\end{proof}

The preceding lemma shows that for most choices of $\alpha,\beta$, either both $g_1$ and $g_2$ act as if they were constant, or one acts as if it were constant, and the other acts as if it were Boolean, up to translation. The following lemma, which is the main result of this section, shows that in fact, one is almost zero, and the other is almost Boolean.

\begin{lemma} \label{lem:decomposition}
 Suppose $(X,Y)$ are typical restrictions. Then either $g_1$ is $(3\error^{1/7},19\error^{1/7})$-almost close to zero and $g_2$ is $(4\error^{1/7},24\error^{1/7})$-almost Boolean, or the same is true with the roles of $g_1$ and $g_2$ reversed.
\end{lemma}
\begin{proof}
 Define
 \begin{align*}
  p_1 &= \Pr_{\alpha,\beta \in T_{X,Y}}[|g_1(\alpha_1)-g_1(\beta_1)| > 2\error^{1/7}], \\
  p_2 &= \Pr_{\alpha,\beta \in T_{X,Y}}[|g_2(\alpha_2)-g_2(\beta_2)| > 2\error^{1/7}].
 \end{align*}
 \autoref{lem:deltas} implies that $p_1 p_2 \leq 8\error^{2/7}$. Thus, either $p_1 \leq 3\error^{1/7}$ or $p_2 \leq 3\error^{1/7}$. Without loss of generality, we may assume that $p_1 \leq 3\error^{1/7}$.

 A simple averaging argument shows that for some choice of $\alpha$, we have
 \[ \Pr_{\beta \in T_{X,Y}}[|g_1(\alpha_1)-g_1(\beta_1)| \leq 2\error^{1/7}] \geq 1 - 3\error^{1/7}. \]
 Therefore, putting $C_1 = g_1(\alpha_1)$, we deduce that $g_1$ is $(3\error^{1/7},2\error^{1/7})$-almost close to~$C_1$.

 Typicality implies that $g = g_1 + g_2$ satisfies $\EE[(|g|-1)^2] \leq \error^{6/7}$. This must be true for some value~$C_2$ of~$g_2$. The function $g_1 - C_1$ is $(3\error^{1/7},2\error^{1/7})$-almost close to zero, and so we can apply \autoref{lem:mean-mode}, with the following parameters:
 \[ \phi \eqdef g_1 - C_1, \quad p \eqdef 3\error^{1/7}, \quad \epsilon \eqdef 2\error^{1/7}, \quad \delta \eqdef \error^{6/7}, \quad C \eqdef C_1 + C_2. \]
 Since $\error$ is small enough, \autoref{lem:mean-mode} implies that
 \[
  |\EE[g_1] - C_1| \leq 6\error^{1/7} + 9\error^{1/7} + 6\sqrt{\frac{\error^{6/7}}{1-3\error^{1/7}}} =
  15\error^{1/7} + O(\error^{3/7}) \leq 16\error^{1/7}.
 \]
 On the other hand, by typicality, $|\EE[g_1]| \leq \error^{1/7}$. Therefore
 \[ |C_1| \leq 17\error^{1/7}. \]
 We conclude that $g_1$ is $(3\error^{1/7},19\error^{1/7})$-almost close to zero.

 We now turn our gaze to $g_2$.
 \autoref{lem:deltas} implies that with probability at least $1-8\error^{2/7}$ over the choice of $\alpha,\beta$, it holds that $|g_2(\alpha_2)-g_2(\beta_2)| \in \upto{\{0,2\}}{2\error^{1/7}}$. A simple averaging argument shows that for some choice of $\alpha$, it holds that
 \[ \Pr_{\beta \in T_{X,Y}} [|g_2(\alpha_2)-g_2(\beta_2)| \notin \upto{\{0,2\}}{2\error^{1/7}}] \leq 8\error^{2/7}. \]
 Let $C_3 = g_2(\alpha_2)$. Then $g_2$ is concentrated on the three values $\{C_3-2,C_3,C_3+2\}$. Another application of \autoref{lem:deltas} will show that it is actually concentrated either on $\{C_3-2,C_3\}$ or on $\{C_3,C_3+2\}$.
 Define
 \begin{align*}
  q_1 &= \Pr_{\beta \in T_{X,Y}} [g_2(\beta_2) \in \upto{C_3+2}{2\error^{1/7}}], \\
  q_2 &= \Pr_{\gamma \in T_{X,Y}} [g_2(\gamma_2) \in \upto{C_3-2}{2\error^{1/7}}].
 \end{align*}
 When $g_2(\beta_2) \in \upto{C_3+2}{2\error^{1/7}}$ and $g_2(\gamma_2) \in \upto{C_3-2}{2\error^{1/7}}$, we have $|g_2(\beta_2) - g_2(\gamma_2)| \in \upto{4}{4\error^{1/7}}$. In particular, since $\error$ is small enough, in this case $\beta_2,\gamma_2$ satisfy none of the options presented by \autoref{lem:deltas}. Hence, we must have $q_1 q_2 \leq 8\error^{2/7}$. Therefore, either $q_1 \leq 3\error^{1/7}$ or $q_2 \leq 3\error^{1/7}$. Without loss of generality, we may assume that $q_1 \leq 3\error^{1/7}$. Putting $C_4 = C_3 - 1$, we conclude that $g_2 - C_4$ is $(4\error^{1/7},2\error^{1/7})$-almost Boolean. (Here, we used the estimate $3\error^{1/7} + 8\error^{2/7} \leq 4\error^{1/7}$, true since $\error$ is small enough.) Our task is now to show that $C_4$ is close to zero.

 Since $g_1$ is $(3\error^{1/7},19\error^{1/7})$-almost close to zero, it follows that $g - C_4$ is $(7\error^{1/7},21\error^{1/7})$-almost Boolean. By typicality, $g$ is $(\error^{4/7},\error^{1/7})$-almost Boolean. Therefore, with probability at least $1-7\error^{1/7}-\error^{4/7} \geq 1-8\error^{1/7}$ over the choice of $\pi \in T_{X,Y}$,
 \begin{equation} \label{eq:simul-bool}
  g(\pi) \in \upto{\{C_4 \pm 1\}}{21\error^{1/7}} \text{ and } g(\pi) \in \upto{\{\pm 1\}}{\error^{1/7}}. 
 \end{equation}
 Suppose that $\pi_+ \in T_{X,Y}$ satisfies~\eqref{eq:simul-bool} with $g(\pi_+)$ $\error^{1/7}$-close to~$1$. Then either $C_4$ is $22\error^{1/7}$-close to zero, or it is $22\error^{1/7}$-close to~$2$.
 Similarly, if $\pi_- \in T_{X,Y}$ satisfies~\eqref{eq:simul-bool} with $g(\pi_-)$ $\error^{1/7}$-close to~$-1$, then either $C_4$ is $22\error^{1/7}$-close to zero, or it is $22\error^{1/7}$-close to~$-2$. Since $\error$ is small enough, if such permutations $\pi_+,\pi_-$ exist, then we can conclude that $|C_4| \leq 22\error^{1/7}$. That would complete the proof of the lemma.

 It remains to rule out the case that for all permutations satisfying~\eqref{eq:simul-bool}, $g(\pi)$ has the same sign. That would imply that $g$ is $(8\error^{1/7},\error^{1/7})$-almost close to~$L$, where $L \in \{ \pm 1 \}$. We apply \autoref{lem:mean-mode}, with the following parameters:
 \[ \phi \eqdef g - C, \quad p \eqdef 8\error^{1/7}, \quad \epsilon \eqdef \error^{1/7}, \quad \delta \eqdef \error^{6/7}, \quad C \eqdef L. \]
 Since $\error$ is small enough, the lemma implies that
 \[
  |\EE[g]-L| \leq 3\error^{1/7} + 24\error^{1/7} + 6\sqrt{\frac{\error^{6/7}}{1-9\error^{1/7}}} \leq 28\error^{1/7}.
 \]
 Therefore $\EE[g]$ is $28\error^{1/7}$-close to $L$. On the other hand, typicality implies that $\EE[g]$ is $2\error^{1/7}$-close to zero. We deduce that $L$ is $30\error^{1/7}$-close to zero. Since $L \in \{ \pm 1 \}$ and $\error$ is small enough, this is a contradiction.
\end{proof}

\subsection{Random partitions} \label{sec:partition}
Subsection~\ref{sec:decomposition} deals with random restrictions $(X,Y)$. The main result, \autoref{lem:decomposition}, shows that with large probability, in the decomposition $g(X,Y) = g_1(X,Y) + g_2(X,Y)$, one of the functions is almost constant, and the other is almost Boolean. In this subsection, we switch the order of the random choices, and deduce a property of random permutations.

We will need the following classical theorem due to Esseen~\cite{Esseen}. For a modern proof, see~\cite[4.1.b]{ProbIneq}. Note we require Esseen's version, rather than Berry's slightly weaker result~\cite{Berry}.
\begin{theorem}[Berry-Esseen]
 Let $X_1,\ldots,X_n$ be independent random variables with finite third moments, and let $S$ be their sum. Define
 \[ \psi = \frac{\sum_{i=1}^n \EE[|X_i-\EE X_i|^3]}{\left(\sum_{i=1}^n \VV[X_i] \right)^{3/2}}. \]
 Let $N$ be a normal random variable with the same mean and variance as $S$. Then $S$ and $N$ are $C_0\psi$-close in distribution, where $C_0 < 1$ is an absolute constant. In other words, for every $t \in \RR$,
 \[ |\Pr[S < t] - \Pr[N < t]| \leq C_0 \psi.\]
\end{theorem}

Before stating the results, we need some definitions. If $(X,Y) \sim \rand$, then the marginal distribution of $X$ is $\randx$, the uniform distribution on the power set of $[n]$. With slight abuse of terminology, we call the subset $X \subset [n]$ a \emph{partition}, as it will correspond to the genuine partition $(X,X^c)$.

For a permutation $\pi \in S_n$ and a partition $X \subset [n]$, define
\[ P_1 := \sum_{i \in X} a_{i\pi(i)}, \quad P_2 := \sum_{i \in \overline{X}} a_{i\pi(i)}. \]
We say that a partition $X$ is {\em good for $\pi$} if either $P_1$ is $25\error^{1/7}$-close to zero and $P_2$ is $25\error^{1/7}$-close to $\pm 1$, or the same is true with the roles of $P_1$ and $P_2$ reversed. Otherwise, we say that $X$ is {\em bad for $\pi$}.
We say that the permutation $\pi \in S_n$ is \emph{good} if with probability at least $4/5$, a random partition $X$ is good for $\pi$. Otherwise, we say that $\pi$ is {\em bad}.

The following lemma shows that most permutations are good.

\begin{lemma} \label{lem:decomposition:perm}
 With probability at least $1-50 \error^{1/7}$, a random permutation $\pi \in S_n$ is good.
\end{lemma}
\begin{proof}
 By \autoref{lem:typical-restriction}, a restriction $(X,Y) \sim \mathcal{R}$ is typical with probability at least $1-3\error^{1/7}$. Suppose $(X,Y)$ is typical. Choose a uniform random permutation $\pi \in T_{X,Y}$. \autoref{lem:decomposition} shows that with probability at least $1-7\error^{1/7}$, $X$ is good for $\pi$. Hence, if we choose a restriction $(X,Y) \sim \rand$ and a permutation $\pi \in T_{X,Y}$ uniformly at random, then $X$ is good for $\pi$ with probability at least $1-10\error^{1/7}$.

 Given $X$, the sets $T_{X,Y}$ partition $S_n$. Therefore, the permutation $\pi$ chosen in the process above is chosen uniformly at random from $S_n$. Furthermore, by definition, the marginal distribution of $X$ is $\randx$. Therefore, if we first choose a permutation $\pi \in S_n$ uniformly at random, and then we choose $X \sim \randx$, then $X$ is good for $\pi$ with probability at least $1-10\error^{1/7}$. Thus, the average probability (over $\pi \in S_n$) that a random partition is bad is at most $10\error^{1/7}$:
 $$\underset{\pi \in S_n}{\mathbb{E}} \left[\Pr_{X \sim U(2^{[n]})}[X \textrm{ is bad for }\pi] \right] \leq 10 \error^{1/7}.$$
 Markov's inequality now implies that the probability that $\pi$ is bad is at most $50\error^{1/7}$:
 $$\Pr_{\pi \in S_n} \left[\Pr_{X \sim U(2^{[n]})}[X \textrm{ is bad for }\pi] > 1/5\right] < \frac{10\error^{1/7}}{1/5}=50 \error^{1/7}.$$
\end{proof}

The next lemma shows that if $\pi$ is a good permutation, then the generalized diagonal $a_{i\pi(i)}$ corresponding to $\pi$ has a special structure: one of its elements is `large', and the rest are `small'. This is, essentially, a consequence of the main statement of \cite{FKN}, but, for the sake of being self-contained, we give a full proof.

\begin{lemma} \label{lem:diagonal}
 Suppose $\pi \in S_n$ is a good permutation. Then for some $m \in [n]$, $|a_{m\pi(m)}|$ is $50\error^{1/7}$-close to~$\pm 1$, and for $i \neq m$, $|a_{i\pi(i)}| \leq 50\error^{1/7}$.
\end{lemma}
\begin{proof}
The proof is inspired by one of the proofs in~\cite{FKN}. Considering what happens when an element `switches sides' allows us to group the elements $a_{i\pi(i)}$ into two groups: `small' elements (close to zero) and `large' elements (close to $\pm 1$). Similar considerations show that there can be at most one large element. The crucial part is showing that not all elements can be small. Indeed, in this case, the sum $P_1$ in the definition of goodness is approximately normal, and so it cannot be concentrated on the two values $\{0,1\}$ or $\{0,-1\}$. The formal proof is as follows.

\medskip

 Define $s_i = a_{i\pi(i)}$. Since $\pi$ is good, $S_0 := \sum_{i=1}^n s_i$ is $50\error^{1/7}$-close to either~$1$ or~$-1$. It cannot be close to both, since $\error$ is small enough. Choose $K \in \{ \pm 1\}$ so that $S_0$ is $50\error^{1/7}$-close to~$K$.

 Define $T(X) = \sum_{i \in X} s_i$. Let $X \sim \randx$, and put $T = T(X)$, so that $T$ is also a random variable. Note that $T = S_0/2 + \sum_{i=1}^n W_i$, where $W_i = s_i (\charf{i \in X} - 1/2)$. Clearly, $X$ is a good partition for $\pi$ if and only if $\overline{X}$ is a good partition for $\pi$. Since $X$ and $\overline{X}$ are equidistributed, given that $X$ is good for $\pi$, $T$ is $25\error^{1/7}$-close to zero with probability $1/2$, and $25\error^{1/7}$-close to $K$ with probability $1/2$. We conclude that with probability at least $2/5$, $T$ is $25\error^{1/7}$-close to zero, and with probability at least $2/5$, $T$ is $25\error^{1/7}$-close to~$K$.

 Consider any $s_i$. Since $2 \cdot 1/5 < 1$ (here, $1/5$ is an upper bound on the probability that a random partition is bad for $\pi$), there is some choice of $Y \subset [n]\setminus \{i\}$ such that both $Y$ and $Y \cup \{ i \}$ are good for $\pi$. Since $|T(Y) - T(Y\cup\{i\})| = |s_i|$, necessarily either $|s_i| \leq 50\error^{1/7}$ ($s_i$ is \emph{small}) or $|s_i|$ is $50\error^{1/7}$-close to $|K| = 1$ ($s_i$ is \emph{large}).

 We claim that not all the $s_i$ can be small. Assume, for the sake of contradiction, that $|s_i| \leq 50\error^{1/7}$ for all~$i$. Applying Berry-Esseen with $X_i=W_i$ shows that $T$ is $\psi$-close in distribution to a normal distribution $N \sim \normal(S_0/2,\sigma^2)$, where
\[
\sigma^2 = \frac{1}{4} \sum_{i=1}^n s_i^2, \quad \psi = \frac{\sum_{i=1}^n |s_i|^3}{\left(\sum_{i=1}^n s_i^2\right)^{3/2}}. 
\]
 The upper bound on $|s_i|$ implies that $|s_i|^3 \leq 50\error^{1/7} s_i^2$, and so
 \[
  \psi \leq \frac{50\error^{1/7} \sum_{i=1}^n s_i^2}{\left(\sum_{i=1}^n s_i^2\right)^{3/2}} =
  \frac{50\error^{1/7}}{\sqrt{\sum_{i=1}^n s_i^2}} = \frac{25\error^{1/7}}{\sigma}.
 \]
 We now obtain a lower bound on $\sigma$. With probability at least $4/5$, $T$ is $25\error^{1/7}$-close to zero or to $K$, and so its distance from its mean $S_0/2$ is at least $1/2 - 50\error^{1/7}$. Hence
 \[
  \sigma^2 = \VV[T] \geq \frac{4}{5} \left(\frac{1}{2} - 50\error^{1/7}\right)^2 = \Omega(1).
 \]
 This shows that $\psi = O(1)$. Concretely, when $\error^{1/7}$ is small enough, $\psi \leq 1/10$.

 For every interval $I$, Berry-Esseen shows that $|\Pr[T \in I] - \Pr[N \in I]| < 2\psi$. We consider three intervals: $I_1 = \upto{0}{25\error^{1/7}}$, $I_2 = \upto{K}{25\error^{1/7}}$, and $I_3$ is the interval `in between': when $K = 1$, $I_3 = (25\error^{1/7},1-25\error^{1/7})$, and when $K = -1$, $I_3 = (-1+25\error^{1/7},-25\error^{1/7})$. By assumption,
 \begin{align*}
 \Pr[T \in I_1] &\geq 2/5, & \Pr[T \in I_2] &\geq 2/5, & \Pr[T \in I_3] &\leq 1/5. \\
 \intertext{Since $2\psi \leq 1/5$, we deduce that}
 \Pr[N \in I_1] &\geq 1/5, & \Pr[N \in I_2] &\geq 1/5, & \Pr[N \in I_3] &\leq 2/5.
 \end{align*}

 The density of a normal distribution is bitonic (increasing and then decreasing), and so
 \begin{equation} \label{eq:bitonic}
  \frac{\Pr[N \in I_3]}{|I_3|} \geq \min\left(\frac{\Pr[N \in I_1]}{|I_1|}, \frac{\Pr[N \in I_2]}{|I_2|}\right).
 \end{equation}
 We have $|I_1| = |I_2| = 50\error^{1/7}$ and $|I_3| = 1-50\error^{1/7}$. Therefore, the left-hand side of \ref{eq:bitonic} is at most (say) $4/5$, and both terms on the right-hand side are at least (say) $1$ (since $\error$ is small enough), a contradiction.

 Concluding, there must be some $m$ such that $|s_m|$ is $25\error^{1/7}$-close to~$1$. We claim that there cannot be two such indices $m,l$. Suppose, for the sake of contradiction, that both $|s_m|$ and $|s_l|$ are $25\error^{1/7}$-close to~$1$. Since $4\cdot 1/5 < 1$, there is some choice of $Y \subset [n] \setminus \{m,l\}$ such that all of $Y, Y\cup \{m\}, Y\cup \{l\}, Y \cup \{m,l\}$ are good for $\pi$. Since $T(Y \cup \{m\}) = T(Y) + s_m$ and $T(Y \cup \{l\}) = T(Y) + s_l$, we see that $s_m,s_l$ must have the same sign. Since $T(Y \cup \{m,l\}) = T(Y) + s_m + s_l$, this implies that $Y$ and $Y \cup \{m,l\}$ cannot both be good. This contradiction shows that there can be at most one large $s_m$.
\end{proof}

\subsection{Strong lines} \label{sec:strong}
The previous section showed that if we pick a generalized diagonal at random in the matrix $(a_{ij})$, then with probability close to 1, we can designate exactly one element in it as `large'. Corollary \ref{cor:diagonal} restates this formally.

In this subsection, we say that an entry $a_{ij}$ is \emph{large} if $|a_{ij}|$ is $50\error^{1/7}$ close to~$1$. Otherwise, we say it is \emph{small}. Note that, contrary to the usage in \autoref{lem:diagonal}, small elements need not be close to~$0$.
While \autoref{lem:diagonal} allows us to deduce that most of the non-large elements in $(a_{ij})$ are actually close to~$0$, for what follows, it will be enough for us to just maintain a distinction between large elements and non-large elements.

Let $(X,Y)$ be a restriction. Denote by $A[X,Y]$ the submatrix $(a_{ij})_{i \in X, j \in Y}$. We say that a generalized diagonal in $A[X,Y]$ is \emph{good} if it contains exactly one large entry. We say that $(X,Y)$ is \emph{$q$-good} if with probability at least $1 - q$, a random generalized diagonal in $A[X,Y]$ is good.

\begin{corollary} \label{cor:diagonal}
 The restriction $([n],[n])$ is $50\error^{1/7}$-good.
\end{corollary}
\begin{proof}
 Immediate from \autoref{lem:decomposition:perm} and \autoref{lem:diagonal}.
\end{proof}

Our goal is to deduce that $([n],[n])$ has a row or column which contains $(1-O(\error^{1/7}))n$ large entries. The general plan of attack is to prove this by induction on~$n$. We will have a separate argument for small values of $n$, and an inductive argument for large $n$. The latter will use the following lemma, which we will apply with $|X'| = \lfloor |X|/2 \rfloor$.
\begin{lemma} \label{lem:good-half}
 Suppose that $(X,Y)$ is $q$-good for $q < 1/2$. For every $X' \subset X$ there exists $Y' \subset Y$ with $|Y'|=|X'|$, such that either $(X',Y')$ or $(X \setminus X',Y \setminus Y')$ is $q$-good. Similarly, for every $Y' \subset Y$ there exists $X' \subset X$ with $|X'|=|Y'|$, such that either $(X',Y')$ or $(X \setminus X',Y \setminus Y')$ is $q$-good.
\end{lemma}
\begin{proof}
 By symmetry, we need only prove the first statement. Fix $X' \subset X$. Since the sets $(T_{X',Y'}:\ Y' \subset Y,\ |Y'|=|X'|)$ partition $T_{X,Y}$, the fact that $(X,Y)$ is $q$-good implies that for some choice of $Y'$, the probability that a random generalized diagonal corresponding to a permutation in $T_{X',Y'}$ is good is at least $1 - q$. Choose such a $Y'$.

 Let $p_1$ be the probability that a random generalized diagonal in $A[X',Y']$ is good, and let $p_2$ be the probability that a random generalized diagonal in $A[X \setminus X',Y \setminus Y']$ is good. Then $p_1 (1 - p_2) + (1 - p_1) p_2 \geq 1 - q$. We claim that this forces $\max(p_1,p_2) \geq 1 - q$.

 Indeed, let $p_1 = (1+\delta_1)/2$ and $p_2 = (1+\delta_2)/2$, where $|\delta_1|,|\delta_2| \leq 1$. We have
 \[  p_1(1-p_2) + (1-p_1)p_2 = \frac{1-\delta_1\delta_2}{2}. \]
 Since $1 - q > 1/2$, we must have $\delta_1 \delta_2 < 0$. Without loss of generality, we may assume that $\delta_1 > 0$. Since $1 \geq -\delta_2$, we have
 \[ p_1 = \frac{1+\delta_1}{2} \geq \frac{1-\delta_1\delta_2}{2} \geq 1-q,\]
 proving the lemma.
\end{proof}

Let $(X,Y)$ be a restriction. If $i \in X$, we say that row~$i$ is {\em $p$-strong for $(X,Y)$} if at least $(1-p)|Y|$ of the entries $\{a_{ij} : j \in Y\}$ are large. If $j \in Y$, we say that column~$j$ is {\em $p$-strong for $(X,Y)$} if at least $(1-p)|X|$ of the entries $\{a_{ij} : i \in X\}$ are large.

We say that $(X,Y)$ has a \emph{$p$-strong row (resp. column)} if some row (resp. column) is $p$-strong for $(X,Y)$. We say that $(X,Y)$ has a \emph{$p$-strong line} if it has either a $p$-strong row or a $p$-strong column.

Our goal is to show that $([n],[n])$ has a strong line. We start by showing that two strong lines must coincide.
\begin{lemma} \label{lem:two-strong}
 Suppose that $(X,Y)$ is $q$-good. Let $X_1,X_2 \subset X$, and let $Y_1,Y_2 \subset Y$ with $|Y_1|=|X_1|$ and $|Y_2|=|X_2|$. Suppose that $(X_1,Y_1)$ has a $p_1$-strong line, and that $(X_2,Y_2)$ has a $p_2$-strong line. If
 $(1-p_1)|X_1| > 1$, $(1-p_2)|X_2| > 1$ and
 \[
  (1-p_1)(1-p_2) \frac{|X_1| |X_2|}{|X|^2} \geq 4q
 \]
 then the strong lines must be the same (defined by the same row or by the same column).
\end{lemma}
\begin{proof}
 Suppose, for the sake of contradiction, that the two restrictions have different strong lines. Let $L_1 \subset X_1 \times Y_1$ consist of the first $t_1 = \lceil (1-p_1)|X_1| \rceil$ indices of large elements in the strong line of $(X_1,Y_1)$, and let $L_2$ consist of the first $t_2 = \lceil (1-p_2)|X_2| \rceil$ indices of large elements in the strong line of $(X_2,Y_2)$.

 Say that $(i_1,j_1) \in L_1$ and $(i_2,j_2) \in L_2$ \emph{conflict} if either $i_1 = i_2$ or $j_1 = j_2$ (or both). If $L_1$ is row~$i$ and $L_2$ is column~$j$, then an entry on $L_1$ not on column~$j$ never conflicts with an entry on $L_2$ not on row~$i$. Therefore, there are at least $(t_1-1)(t_2-1)$ non-conflicting pairs. If both $L_1$ and $L_2$ are rows (resp. columns), then two entries conflict only if they are on the same column (resp. row). Therefore, the number of non-conflicting pairs is at least $t_1t_2 - \min(t_1,t_2) \geq (t_1-1)(t_2-1)$. 

 For each non-conflicting pair, the probability that a random generalized diagonal in $A[X,Y]$ goes through both entries of the pair is $1/|X|(|X|-1)$. Since these events are all disjoint, it follows that
 \[
  \frac{(t_1-1)(t_2-1)}{|X|^2} < \frac{(t_1-1)(t_2-1)}{|X|(|X|-1)} \leq q.
 \]
 Since $t_1,t_2 \geq 2$, $t_1 - 1 \geq t_1/2$ and $t_2 - 1 \geq t_2/2$. Using $t_1/|X_1| \geq 1-p_1$ and $t_2/|X_2| \geq 1-p_2$, we deduce that
 \[
  (1-p_1)(1-p_2)\frac{|X_1| |X_2|}{|X|^2} < 4q,  
 \]
 contradicting our assumption.
\end{proof}
Note that the conditions $(1-p_1)|X_1| > 1$ and $(1-p_2)|X_2| > 1$ simply guarantee that each strong line has at least two large elements. (If one of the strong lines had only one large element, then it could be contained in the other strong line, and so there would be no contradiction.) 

Our next result says that if there is one strong line, then there cannot be many large entries outside the line. For the proof, we need the simplest case of Bonferroni's inequality.
\begin{classic}[Bonferroni]
 Let $A_1,\ldots,A_h$ be events. Then
 \[ \Pr[A_1 \lor \cdots \lor A_h] \geq \sum_i \Pr[A_i] - \sum_{i < j} \Pr[A_i \land A_j]. \]
\end{classic}

\begin{lemma} \label{lem:strong-scattered}
 Suppose that $(X,Y)$ is $q$-good and has a $p$-strong line. Let $m = |X|$, and let $\varrho = 2q/(1-p)$. If $m \geq 6$, $(1-p)m > 1$, $2\varrho m > 1$ and $\varrho \leq 1/2$, then that line is actually $(q + 3\varrho)$-strong.
\end{lemma}
\begin{proof}
 Without loss of generality, we may assume that the $p$-strong line is row~$i$. Since $(X,Y)$ is $q$-good, a random element in $A[X,Y]$ is large with probability at least $(1-q)/m$. Therefore, $A[X,Y]$ contains at least $(1-q)m$ large entries.

 Suppose that row~$i$ is not $(q+3\varrho)$-strong. Then $A[X,Y]$ contains at least $3\varrho m$ large entries outside row~$i$. \autoref{lem:two-strong} implies that no other line can be $(1-2\varrho)$-strong. Therefore, no column can contain more than $2\varrho m$ large entries. So for any column~$j$, $A[X,Y]$ contains at least $\varrho m$ large entries outside row~$i$ and column~$j$. The probability that a random generalized diagonal in $A[X,Y]$ hits any single one of these, given that it hits a specific large entry in row~$i$, is $1/(m-1)$, and the probability that it hits any two of them is at most $1/(m-1)(m-2)$. Therefore, Bonferroni's inequality implies that the probability that a generalized diagonal in $A[X,Y]$ contains at least two large elements is at least
 \[
  (1-p) \left( \frac{\varrho m}{m-1} - \frac{(\varrho m + 1)(\varrho m)}{2(m-1)(m-2)} \right) =
  \frac{(1-p)\varrho m}{m-1} \left(1 - \frac{\varrho m + 1}{2(m-2)}\right) > q,
 \]
 since $\varrho m + 1 \leq m/2 + 1 \leq m - 2$. But this probability must be at most $q$, a contradiction.
\end{proof}

The following sequence of lemmas shows the existence of a strong line in $(X,Y)$, given that $(X,Y)$ is $q$-good for $q$ sufficiently small depending on $|X|$.
\begin{lemma} \label{lem:strong:0}
 Suppose that $(X,Y)$ is $q$-good for some $q < \tfrac{1}{m(m-1)}$, where $m = |X|$. Then $(X,Y)$ has a $0$-strong line.
\end{lemma}
\begin{proof}
 Without loss of generality, we may assume that $X = Y = [m]$. Since $q < 1/m$, there must exist a permutation $\pi \in S_m$ such that all generalized diagonals of the form $\{ (i,\pi(i) + j) : i \in [m] \}$ are good. Without loss of generality, we may assume that $\pi$ is the identity, and that $a_{1,1}$ is large. For $j \in [m]$, let $a_{i,i+j}$ be the large entry on the corresponding diagonal. If $i \neq 1$ and $i + j \neq 1$, then a random generalized diagonal in $A[[m],[m]]$ passes through both $a_{1,1}$ and $a_{i,i+j}$ with probability $1/(m(m-1))$, contrary to our assumption. Therefore, all large entries are either on row~$1$ or on column~$1$. If there is an entry $a_{i,1}$ not on row~$1$ and an entry $a_{1,j}$ not on column~$1$, then we again reach a contradiction. It follows that either row~$1$ or column~$1$ consists of large elements only.
\end{proof}

We now improve this result using induction.

\begin{lemma} \label{lem:strong:0a}
 Suppose that $(X,Y)$ is $q$-good for some $q < \tfrac{1}{4m}$, where $m = |X|$. Then $(X,Y)$ has a $0$-strong line.
\end{lemma}
\begin{proof}
 The proof is by induction on~$m$. When $m \leq 5$, the claim follows from \autoref{lem:strong:0}, so suppose that $m \geq 6$.

 Let $X' \subset X$ be an arbitrary subset of size $s = \lfloor m/2 \rfloor$. \autoref{lem:good-half} shows that there exists $Y' \subset Y$ with $|Y'|=|X'|$, such that either $(X',Y')$ is $q$-good or $(X \setminus X',Y \setminus Y')$ is $q$-good. The induction hypothesis implies that either $(X',Y')$ or $(X \setminus X',Y \setminus Y')$ has a $0$-strong line. Similarly, if $Y' \subset Y$ is an arbitrary subset of size $s = \lfloor m/2 \rfloor$, then there exists $X' \subset X$ with $|X'|=|Y'|$ such that either $(X',Y')$ or $(X \setminus X',Y \setminus Y')$ has a 0-strong line.
 
Since $s \geq 3$, we have
 \[
  \left(\frac{s}{m}\right)^2 \geq \frac{9}{49} > \frac{1}{6} \geq \frac{1}{m} > 4q,
 \]
 so \autoref{lem:two-strong} implies that all 0-strong lines arising from different choices of $X'$ or $Y'$ must be defined by the same row or column --- say row~$i$.

 We claim that row~$i$ can have at most one small entry. If row~$i$ has at least two small entries $a_{ij},a_{ik}$, then there exists $Y' \subset Y$ with $|Y|=s$, such that $j \in Y'$ and $k \in Y \setminus Y'$. For any $X' \subset X$ with $|X'|=s$, row~$i$ can be $0$-strong in neither $(X',Y')$ nor $(X\setminus X',Y\setminus Y')$, a contradiction. Thus, row~$i$ has at most one small entry.

 Suppose that row~$i$ has exactly one small entry. Since $q < 1/4m < 1/m$, a random entry in $A[X,Y]$ is large with probability at least $(1-q)/m > 1/m - 1/m^2$, and so there must be at least $m$ large entries. Exactly $m-1$ of these are on row~$i$. Let $a_{kl}$ be another large entry. The probability that a random generalized diagonal passes through both $a_{kl}$ and one of the large entries on row~$i$ is at least
\[ \frac{1}{m} \left(1 - \frac{1}{m-1}\right) \geq \frac{4/5}{m} > \frac{1}{4m},\]
a contradiction. Hence, row~$i$ must be $0$-strong for $A[X,Y]$, completing the proof.
\end{proof}

We now use induction to tackle the case of large $|X|$.
\begin{lemma} \label{lem:strong}
 Suppose that $(X,Y)$ is $q$-good for $q < 1/50$. Then $(X,Y)$ has a $13q$-strong line. 
\end{lemma}
\begin{proof}
 The proof is by induction on $m := |X|$. When $m < 1/(4q)$, the statement of the lemma follows from \autoref{lem:strong:0a}, so suppose that $m \geq 1/(4q) \geq 12$.

 Let $X' \subset X$ be an arbitrary subset of size $s = \lfloor m/2 \rfloor \geq 6$. \autoref{lem:good-half} shows that there exists $Y' \subset Y$ with $|Y'|=|X'|$, such that either $(X',Y')$ is $q$-good or $(\overline{X'},\overline{Y'})$ is $q$-good. The induction hypothesis implies that either $(X',Y')$ or $(X \setminus X',Y \setminus Y')$ has a $13q$-strong line. Similarly, if $Y' \subset Y$ is an arbitrary subset of size $s = \lfloor m/2 \rfloor$, then there exists $X' \subset X$ with $|X'|=|Y'|$ such that either $(X',Y')$ or $(X \setminus X',Y \setminus Y')$ has a $13q$-strong line. 
 
 Since $(1-13q)s > 1$ and $(1-13q)^2 (s/m)^2 \geq 4q$, \autoref{lem:two-strong} implies that all the $13q$-strong lines arising from different choices of $X'$ or $Y'$ must be defined by the same row or column --- say row~$i$.

 We claim that row~$i$ has at most $\lfloor 13qm + 1 \rfloor$ small entries. Indeed, suppose it has at least $\lfloor 13qm + 2 \rfloor$ small entries. Then there exists a subset $Y' \subset Y$ with $|Y'|=s$, such that the $Y'$-part of row~$i$ contains at least $\lfloor 13qs + 1 \rfloor$ small entries and the $Y \setminus Y'$-part of it contains at least $\lfloor 13q(m-s)+1 \rfloor$ small entries. For any $X' \subset X$, row~$i$ is not $13q$-strong in either $(X',Y')$ nor $(X\setminus X',Y\setminus Y')$, a contradiction. Therefore, row~$i$ is a $(13q + 1/m)$-strong line for $(X,Y)$.

 Let $p = 13q + 1/m$. We are going to apply \autoref{lem:strong-scattered}. Clearly $m \geq 12$, and it is easy to check that $(1-p)m > 1$ and $\varrho \leq 1/2$. Slightly more delicately, we have
\[
 2\varrho m = \frac{4qm}{1 - 13q - 1/m} > 4qm \geq 1.
\]
 Hence, by \autoref{lem:strong-scattered}, row~$i$ is $(q+3\varrho)$-strong. Since $p < 1/2$, we have $q + 3\varrho < 13q$, so row~$i$ is $13q$-strong, completing the proof.
\end{proof}

It might seem that the condition $2\varrho m > 1$ is very tight. This will not matter for us, but in fact, one can prove a version of \autoref{lem:strong-scattered} with a weaker condition, at the cost of obtaining a worse guarantee on the strength of the line.

\begin{corollary} \label{cor:strong}
 There exists an $O(\error^{1/7})$-strong line for $([n],[n])$.
\end{corollary}
\begin{proof}
 Follows immediately from \autoref{cor:diagonal} and \autoref{lem:strong}.
\end{proof}

\subsection{Culmination of the proof} \label{sec:culmination}
 
In this section, we will see what \autoref{cor:strong} implies in terms of the original family $\cF$. Without loss of generality, we may assume for the rest of this section that the $O(\error^{1/7})$-strong line whose existence is guaranteed by \autoref{cor:strong} is row~$1$.

What the corollary implicitly says is that the matrix $(a_{ij})$ looks very like the canonical example shown in the introduction:
\canonicalexample
Indeed, the corollary shows that, without loss of generality, the first row consists mainly of elements which are very close to $\pm 1$. Since the line must sum to zero, we know that roughly half of these are close to $1$, and roughly half to $-1$. This information will enable us to deduce that $\mathcal{F}$ is close to a disjoint union of roughly $n/2$ cosets.

For $i,j \in [n]$, we define
$$\tau_{ij} := \frac{|\cF \cap T_{ij}|}{(n-1)!}.$$
By \eqref{eq:a-formula}, we have

\begin{equation} \label{eq:t-formula}
 \tau_{ij} = \frac{1}{2} + \frac{n}{2(n-1)} a_{ij}.
\end{equation}
 
So we have $\tau_{ij} \approx (a_{ij}+1)/2$. More precisely, we have the following.
\begin{lemma} \label{lem:t-rough}
 Each $\tau_{ij}$ is $1/n$-close to $(a_{ij}+1)/2$. If $a_{ij}$ is large, then $\tau_{ij}$ is $26\error^{1/7}$-close to $\{0,1\}$.
\end{lemma}
\begin{proof}
 The formula \eqref{eq:a-formula} for $a_{ij}$ implies that $|a_{ij}| \leq 1$. We have
 \[
 \tau_{ij} = \frac{a_{ij}+1}{2} + \frac{a_{ij}}{2(n-1)}. 
 \]
 The second term has absolute value at most $\tfrac{1}{2(n-1)} \leq \tfrac{1}{n}$, since $n \geq 2$.

 A large entry is $50\error^{1/7}$ close in magnitude to~$\pm 1$, and so $(a_{ij}+1)/2$ is $25\error^{1/7}$-close to $\{0,1\}$. Finally, assumption~\ref{eq:e-assumption} implies that $1/n \leq \error^{1/7}$.
\end{proof}

We are now almost ready to prove our main result.
\begin{lemma} \label{lem:large-entries}
 The number of $\tau_{1i}$ which are $26\error^{1/7}$-close to~$1$ is $O(\error^{1/7})n$-close to $n/2$.
\end{lemma}
\begin{proof}
 Let $N_0$ be the number of $\tau_{1i}$ which are $26\error^{1/7}$-close to~$0$, and let $N_1$ be the number of $\tau_{1i}$ which are $26\error^{1/7}$-close to~$1$. \autoref{lem:a} shows that
 \[T:= \sum_{i=1}^n \tau_{1i} = \frac{n}{2}. \]
 On the other hand,
 \[ (1-26\error^{1/7}) N_1 \leq T \leq 26\error^{1/7} N_0 + (n - N_0) = n - (1-26\error^{1/7}) N_0. \]
 Substituting $T=n/2$, we obtain
 \[ N_0,N_1 \leq (1+O(\error^{1/7})) n/2. \]
 The fact that row 1 is an $O(\error^{1/7})$-strong line implies that $N_0 + N_1 = (1-O(\error^{1/7})) n$, and so
 \[N_1 \geq (1-O(\error^{1/7})) n - N_0 \geq (1-O(\error^{1/7}))n/2. \qedhere \]
\end{proof}

The main result easily follows.
\begin{corollary} \label{cor:main}
 Suppose that $n \geq 4$, and
 $$\frac{1}{n^{7/3}} \leq \epsilon_1 < \epsilon_0,$$
 where $\epsilon_0>0$ is an absolute constant. Let $\cF \subset S_n$ be a family of permutations with size $|\cF| = n!/2$, satisfying
 \[ \EE[(f-f_1)^2] = \error, \]
where $f = 2\chi_\cF-1$, and $f_1$ is the orthogonal projection of $f$ onto $U_1$. Then there exists a family $\cG \subset S_n$ which is a union of $\lfloor n/2 \rfloor$ disjoint 1-cosets, satisfying
 \[ |\cG \triangle \cF| \leq O(\error^{1/7}) n!. \]
\end{corollary}
\begin{proof}
  Lemma \ref{lem:large-entries} implies that the set $S = \{ i : \tau_{1i} \geq 1 - 26\error^{1/7} \}$ has cardinality which is $O(\error^{1/7})n$-close to $n/2$. By assumption, $\epsilon_1 \geq n^{-7/3}$, and therefore
  $$||S| - \lfloor n/2 \rfloor| \leq ||S|-n/2|+1/2 \leq O(\epsilon_1^{1/7})n+1/2 \leq O(\epsilon_1^{1/7})n+n^{2/3} \leq O(\epsilon_1^{1/7})n.$$
  Define
  \[ \cG' = \bigcup_{i \in S} T_{1i}. \]
  Note that the $T_{1i}$ are pairwise disjoint, and so $|\cG'| = (n-1)!|S|$. By the definition of $S$, we have
  \[
   |\cF \cap \cG'| \geq (1-26\error^{1/7})(n-1)!|S| = (\tfrac{1}{2} - O(\error^{1/7}))n!.
  \]
  It follows that
  \begin{align*}
   |\cF \triangle \cG'| & = |\cF| + |\cG'| - 2|\cF \cap \cG| \\
   &\leq n!/2 + 
   (1 + O(\error^{1/7}))n!/2 - 2(\tfrac{1}{2} - O(\error^{1/7}))n!\\
   & = O(\error^{1/7}) n!.
  \end{align*}
  By adding or deleting
  $$||S|-\lfloor n/2 \rfloor| = O(\epsilon_1^{1/7})n$$
  $T_{1i}$'s from $\mathcal{G}'$, we may produce a family $\mathcal{G} \subset S_n$ which is a union of $\lfloor n/2 \rfloor$ disjoint 1-cosets, and satisfies $|\cF \triangle \cG| = O(\error^{1/7}) n!$, completing the proof.
\end{proof}

\section{Proof in the general case} \label{sec:general-case}
Up until now, we have only discussed the case $|\cF| = n!/2$. Much of the argument remains intact for general values of $c = |\cF|/n!$, although an additional argument is required in the culmination of the proof. Also, $\errub$ will now depend upon $c$. More concretely, let
\begin{equation} \label{eq:c-assumption}
\eta = \min\{c,1-c\}.
\end{equation}
As we shall see below, for the proof to go through, we will need $\errub = O(\eta^7)$. Indeed, we shall make the following assumption.

 \begin{equation} \label{eq:e-assumption:general} \tag{\ref{eq:e-assumption}'}
 \frac{1}{n^{7/3}} \leq \error \leq c_0 \eta^7, 
 \end{equation}

To explore all of these issues, let us follow the existing proof and see how it adapts for arbitrary $c$.

\paragraph{Matrix representation (\autoref{sec:matrix})}

The coefficients $a_{ij}$ for $c = 1/2$ were defined so that the following holds.
\begin{equation} \label{eq:f1-a}
f_1 = \sum_{i,j} a_{ij} T_{ij}.
\end{equation}
As we remarked in the proof of \autoref{lem:a}, our definition of $a_{ij}$ can be derived from this formula via Fourier inversion. For arbitrary $c$, the corresponding definition is
\begin{equation} \label{eq:a-formula:general} \tag{\ref{eq:a-formula}'}
 a_{ij} = (n-1) \langle f, T_{ij} \rangle - \frac{n-2}{n} (2c-1).
\end{equation}
Under this definition, \eqref{eq:f1-a} holds. Straightforward calculations yield the following updated versions of \autoref{lem:a} and \autoref{lem:a-l2}.

{
\renewcommand{\thelemma}{\ref{lem:a}'}
\addtocounter{lemma}{-1}
\begin{lemma} \label{lem:a:general} 
 We have
\[ f_1 = \sum_{i,j} a_{ij} T_{ij}. \]
 Furthermore,
\[
 \sum_j a_{ij} = 2c-1 \quad \forall i \in [n],\quad \sum_i a_{ij} = 2c-1\quad \forall j \in [n].
\]
 For each permutation $\pi$, we have:
\[
 f_1(\pi) = \sum_i a_{i\pi(i)}.
\]
\end{lemma}
}

{
\renewcommand{\thelemma}{\ref{lem:a-l2}'}
\addtocounter{lemma}{-1}
\begin{lemma} \label{lem:a-l2:general}
 We have
\[
 \sum_{i,j} a_{ij}^2 = (n-1)(1-\error) - (n-2)(2c-1)^2.
\]
\end{lemma}
}

\paragraph{Random restrictions (\autoref{sec:restrictions})}

The proof of \autoref{lem:e-moments} becomes more cumbersome. Curiously enough, the variance of the random variable in question is actually maximized when $c = 1/2$, and so the bound on the variance holds true for arbitrary $c$. Here is the updated version.

{
\renewcommand{\thelemma}{\ref{lem:e-moments}'}
\addtocounter{lemma}{-1}
\begin{lemma} \label{lem:e-moments:general}
 If $(X,Y)$ is a restriction, define $m(X,Y) = \EE_{T_{X,Y}}[g_1(X,Y)]$. Let $(X,Y) \sim \rand$. Then $\EE_{\mathcal{R}}[m] = c - 1/2$ and $\VV_{\mathcal{R}}[m] \leq 1/2n$. 
\end{lemma}
\begin{proof}
 We only give the exact formula for $\VV_{\mathcal{R}}[m]$:
\[
 \VV_{\mathcal{R}}[m] = \frac{n+1}{4n(n-1)} - \frac{(2c-1)^2}{2n(n-1)} \leq \frac{1}{2n}. \qedhere
\]
\end{proof}
}

The proof of \autoref{lem:typical-restriction} remains the same, adjusting for the general value of $\EE_{\mathcal{R}}[m]$.
{
\renewcommand{\thelemma}{\ref{lem:typical-restriction}'}
\addtocounter{lemma}{-1}
\begin{lemma} \label{lem:typical-restriction:general}
 Let $(X,Y) \sim \rand$. With probability at least $1-3\error^{1/7}$, $(X,Y)$ satisfies the following properties:
\begin{enumerate}[(a)]
 \item $g(X,Y)$ is $(\error^{4/7},\error^{1/7})$-almost Boolean.
 \item $\EE[g_1(X,Y)]$ and $\EE[g_2(X,Y)]$ are $\error^{1/7}$-close to $c-1/2$.
 \item $\EE[(|g(X,Y)|-1)^2] \leq \error^{6/7}$.
\end{enumerate}
\end{lemma}
}

We redefine a \emph{typical restriction} as one satisfying these updated properties.

\paragraph{Decomposition under random restriction (\autoref{sec:decomposition})}

\autoref{lem:mean-mode} and \autoref{lem:deltas} are not affected by the value of $c$. The value of $c$ comes into play in the main result of this section, \autoref{lem:decomposition}, in two ways. Firstly, the result is affected by the change in \autoref{lem:e-moments}. Secondly, there is a hidden dependence of $\errub$ on $c$, namely
\begin{equation} \label{eq:errub:decomposition}
 \eta > 15 \error^{1/7}. 
\end{equation}

{
\renewcommand{\thelemma}{\ref{lem:decomposition}'}
\addtocounter{lemma}{-1}
\begin{lemma} \label{lem:decomposition:general}
 Suppose that $(X,Y)$ is a typical restriction. Then either $g_1$ is $(3\error^{1/7},19\error^{1/7})$-almost close to $c - 1/2$ and $g_2 + c - 1/2$ is $(4\error^{1/7},24\error^{1/7})$-almost Boolean, or the same is true with the roles of $g_1$ and $g_2$ reversed.
\end{lemma}
\begin{proof}
 At the very end of the proof, we need to rule out possibility that for all permutations $\pi \in S_n$ satisfying
\begin{equation} \tag{\ref{eq:simul-bool}}
  g(\pi) \in \upto{\{C_4 \pm 1\}}{21\error^{1/7}} \text{ and } g(\pi) \in \upto{\{\pm 1\}}{\error^{1/7}}, 
\end{equation}
the sign of $g(\pi)$ is the same. This would imply that $g$ is $(8\error^{1/7},\error^{1/7})$-almost close to $L \in \{\pm 1\}$. Applying \autoref{lem:mean-mode}, this in turn would imply that $\EE[g]$ is $28\error^{1/7}$-close to $L$. On the other hand, typicality implies that $\EE[g]$ is $2\error^{1/7}$-close to $2c-1$. In order to obtain a contradiction, we need to assume that $2c-1$ is not $30\error^{1/7}$-close to $\pm 1$. This is equivalent to~\eqref{eq:errub:decomposition}.
\end{proof}
}

\paragraph{Random partitions (\autoref{sec:partition})}

At the beginning of this section, we defined the concept of a \emph{good partition}, which we now need to update. For a permutation $\pi \in S_n$ and a partition $X \subset [n]$, define
\[ P_1 := \sum_{i \in X} a_{i\pi(i)}, \quad P_2 := \sum_{i \in \overline{X}} a_{i\pi(i)}. \]
We say that the partition $X$ is {\em good for $\pi$} if either $P_1$ is $25\error^{1/7}$-close to $c-1/2$ and $P_2$ is $25\error^{1/7}$-close to $\{-c-1/2,3/2-c\}$, or the same is true with the roles of $P_1$ and $P_2$ reversed. We say that the permutation $\pi \in S_n$ is \emph{good} if with probability at least $4/5$, a random partition $X \sim \randx$ is good for $\pi$ (this is the same definition as before). With the updated definition, \autoref{lem:decomposition:perm} remains the same.

As for \autoref{lem:diagonal}, apart from slightly updating the statement, there is also a hidden dependence of $\errub$ upon $c$, namely
\begin{equation} \label{eq:errub:partition}
 \eta = \Omega(\error^{1/7}).
\end{equation}

{
\renewcommand{\thelemma}{\ref{lem:diagonal}'}
\addtocounter{lemma}{-1}
\begin{lemma} \label{lem:diagonal:general}
 Suppose that $\pi \in S_n$ is a good permutation. Then for some $m \in [n]$, $|a_{m\pi(m)}|$ is $50\error^{1/7}$-close to~$\{2c,2(1-c)\}$, and for $i \neq m$, $|a_{i\pi(i)}| \leq 50\error^{1/7}$.
\end{lemma}
\begin{proof}
 We redefine `large' elements as those which are $50\error^{1/7}$-close in magnitude to $\{2c,2(1-c)\}$. Analyzing what happens when a single $s_i$ switches over, we deduce as in the original proof that each $s_i$ is either small or large.

 For the Berry-Esseen argument, we need a lower bound on $\sigma^2$. In the original proof, we deduced such a bound from the fact that with probability at least $4/5$, it holds that $|T-S_0/2| \geq 1/2 - 50\error^{1/7}$. The same argument shows that with probability at least $4/5$, it holds that $|T-S_0/2| \geq \eta - 50\error^{1/7}$, and therefore
\[ \sigma^2 \geq \frac{4}{5} (\eta - 50\error^{1/7})^2. \]
 This implies that
\[ \psi = \frac{O(\error^{1/7})}{\eta - 50\error^{1/7}}, \]
 where the implied constant does not depend upon $c$. Condition~\eqref{eq:errub:partition} guarantees that (say) $\psi \leq 1/10$.

 The intervals $I_1,I_2$ retain their length, while for $I_3$ we get the guarantee
\[ |I_3| \geq 2\eta - 50\error^{1/7}. \] 
 Recall inequality~\eqref{eq:bitonic}, from which we would like to derive a contradiction:
 \begin{equation} \tag{\ref{eq:bitonic}}
  \frac{\Pr[N \in I_3]}{|I_3|} \geq \min\left(\frac{\Pr[N \in I_1]}{|I_1|}, \frac{\Pr[N \in I_2]}{|I_2|}\right).
 \end{equation}
 The left-hand side is at most (roughly) $1/(5\eta)$, while the right-hand side is $\Omega(\error^{-1/7})$. We get a contradiction if $1/\eta = O(\error^{-1/7})$, which is the same condition as~\eqref{eq:errub:partition}.

 The rest of the proof goes through without change.
\end{proof}
}

\paragraph{Strong lines (\autoref{sec:strong})}

As we mentioned in the introduction, this part is almost completely independent of the rest of the proof. All we have to do is redefine a \emph{large entry} so that it conforms to the specification of \autoref{lem:diagonal:general}, that is, $|a_{ij}|$ is $50\error^{1/7}$-close to $\{2c,2(1-c)\}$. With this small change, all the results in this section carry through.

\paragraph{Culmination of the proof (\autoref{sec:culmination})}

This section requires a small overhaul. Whereas for $c = 1/2$, a large element was always close to $\pm 1$, now all we know is that it is close in magnitude to $\{2c,2(1-c)\}$. Its actual value is therefore close to one of the values $\{2c,2(1-c),-2c,-2(1-c)\}$. Defining $\tau_{ij}$ as before, this means that $\tau_{ij}$ is close to one of the values $\{0,1,2c,2c-1\}$. Of these, one is always outside $[0,1]$ and so cannot occur, and one is a `medium' value, 
\[
 \gamma := \begin{cases} 2c & \text{if } c < 1/2, \\ 2c-1 & \text{if } c > 1/2 \end{cases},
\]
lying inside the interval $(0,1)$. An additional argument is needed to show that such medium values do not actually occur in large quantities on the strong line.

We first rearrange the formula~\eqref{eq:a-formula:general}:
\begin{equation} \label{eq:t-formula:general} \tag{\ref{eq:t-formula}'}
 \tau_{ij} := \frac{|\cF \cap T_{ij}|}{(n-1)!} = \frac{2(n-2)c+1}{2(n-1)} + \frac{n}{2(n-1)} a_{ij}.
\end{equation}
Roughly, we have $\tau_{ij} \approx a_{ij}/2 + c$. More precisely, we have the following analogue of \autoref{lem:t-rough}.

{
\renewcommand{\thelemma}{\ref{lem:t-rough}'}
\addtocounter{lemma}{-1}
\begin{lemma} \label{lem:t-rough:general}
 Each $\tau_{ij}$ is $2/n$-close to $a_{ij}/2 + c$.

 If $a_{ij}$ is large, then $\tau_{ij}$ is $26\error^{1/7}$-close to $\{0,1,\gamma\}$.
\end{lemma}
}

As before, without loss of generality, we may assume that row 1 is the strong line. Before proving the analogue of \autoref{lem:large-entries}, we need to show that for any two `reasonable' large entries on row 1, either both are close to $\{0,1\}$, or both are close to $\gamma$. Since most entries turn out to be `reasonable', this implies a dichotomy: either most entries are close to $\{0,1\}$, or most are close to $\gamma$. Since the row sums to roughly $cn$, the second case cannot occur.

For $j \in [n]$, let $r(j)$ be the probability that a random generalized diagonal passing through $a_{1j}$ is good. We say that an entry $a_{1j}$ is \emph{reasonable} if $a_{1j}$ is large, $r(j) \geq 4/5$ and $g(\{1\},\{j\})$ is $(1/5,\error^{1/7})$-almost Boolean. (Recall that $g(\{1\},\{j\})$ is the function $f_1$ restricted to permutations in $T_{1j}$.)

\begin{lemma} \label{lem:middle-value}
 Assume that $\gamma$ is $156\error^{1/7}$-far from $\{0,1\}$. Let $j,k \in [n]$ be such that $a_{1j}$ and $a_{1k}$ are reasonable. Either both $\tau_{1j}$ and $\tau_{1k}$ are $26\error^{1/7}$-close to $\gamma$, or neither of them are.
\end{lemma}
\begin{proof}
If $\pi \in T_{1j}$ then $(jk)\pi \in T_{1k}$. Since $4 \cdot 1/5 < 1$, there exists a permutation $\pi \in T_{1j}$ such that both $f_1(\pi)$ and $f_1((jk)\pi)$ are $\error^{1/7}$-close to $\pm 1$, and both $\pi$ and $(jk)\pi$ are good. Let $i = \pi^{-1}(k)$. Note that
 \[
  (f_1(\pi) - a_{ik}) - (f_1((jk)\pi) - a_{ij}) = a_{1j} - a_{1k} =
  \frac{2(n-1)}{n} (\tau_{1j} - \tau_{1k}).
 \]
 Since both $a_{ik}$ and $a_{ij}$ are small, the left-hand side is $102\error^{1/7}$-close to $\{0,\pm 2\}$, and therefore $\tau_{1j} - \tau_{1k}$ is $102\error^{1/7}$-close to $\{0, \pm n/(n-1)\}$. Assumption~\eqref{eq:e-assumption}' implies that $2\error^{1/7} > 2/n > 1/(n-1)$, and so $\tau_{1j} - \tau_{1k}$ is $104\error^{1/7}$-close to $\{0, \pm 1\}$.

 Suppose for a contradiction that $\tau_{1j}$ is $26\error^{1/7}$-close to $\gamma$, and that $\tau_{1k}$ is $26\error^{1/7}$-close to $b \in \{0,1\}$. Then $\gamma$ is $156\error^{1/7}$-close to $b + \{0,\pm1\} \in \{-1,0,1,2\}$. Since $\gamma \in (0,1)$, $\gamma$ must be $156\error^{1/7}$-close to $\{0,1\}$, contradicting the assumption of the lemma.
\end{proof}

Then following lemma says that most of the entries $a_{1j}$ are reasonable.

\begin{lemma} \label{lem:reasonable}
 The probability that $a_{1j}$ is not reasonable for a uniform random $j \in [n]$ is $O(\error^{1/7})$.
\end{lemma}
\begin{proof}
\autoref{cor:strong} shows that the probability that $a_{1j}$ is not large is $O(\error^{1/7})$.

\autoref{cor:diagonal} shows that $\EE[1-r(j)] \leq 50\error^{1/7}$, and so, by Markov's inequality,
$$\Pr_{j \in [n]}[1-r(j) > 1/5] < \frac{50\error^{1/7}}{1/5} = 250\error^{1/7}.$$

Notice that $f_1$ is $(\error^{5/7},\error^{1/7})$-almost Boolean, since otherwise $\EE[(f_1-f)^2] \geq \EE[(|f_1|-1)^2] > \error^{5/7} \error^{2/7} = \error$. Therefore, by Markov's inequality,
$$\Pr_{j \in [n]}[g(\{1\},\{j\})\textrm{ is not }(1/5,\error^{1/7})\textrm{-almost Boolean}] \leq \frac{\error^{5/7}}{1/5} = 5\error^{5/7}.$$
The lemma follows from a union bound.
\end{proof}

We can now prove the analogue of \autoref{lem:large-entries}.
{
\renewcommand{\thelemma}{\ref{lem:large-entries}'}
\addtocounter{lemma}{-1}
\begin{lemma} \label{lem:large-entries:general}
 The number of $\tau_{1i}$ which are $51\error^{1/7}$-close to~$1$ is $O(\error^{1/7})n$-close to $cn$.
\end{lemma}
}
\begin{proof}
 There are two cases, depending on whether $\gamma$ is $156\error^{1/7}$-close to $\{0,1\}$ or not.

Suppose first that $\gamma$ is $156\error^{1/7}$-close to $\{0,1\}$. Let $N_0$ be the number of $\tau_{1i}$ which are $182\error^{1/7}$-close to~$0$, and let $N_1$ be the number of $\tau_{1i}$ which are $182\error^{1/7}$-close to~$1$. \autoref{lem:a:general} shows that
 \[ T := \sum_{i=1}^n \tau_{1i} = cn. \]
 On the other hand,
 \[ (1-182\error^{1/7}) N_1 \leq T \leq 182\error^{1/7} N_0 + (n - N_0) = n - (1-182\error^{1/7}) N_0. \]
 Substituting in $T=cn$ gives:
 \begin{align*}
  N_0 &\leq (1+O(\error^{1/7})) (1-c)n, \\
  N_1 &\leq (1+O(\error^{1/7})) cn.
 \end{align*}
 \autoref{cor:strong} and \autoref{lem:t-rough:general} together show that $N_0 + N_1 = (1-O(\error^{1/7})) n$, and so
 $$ N_1 \geq (1-O(\error^{1/7}))n - N_0 \geq (c-O(\error^{1/7}))n.$$
 This completes the proof when $\gamma$ is $156\error^{1/7}$-close to $\{0,1\}$.

Suppose next that $\gamma$ is $156\error^{1/7}$-far from $\{0,1\}$, so that \autoref{lem:middle-value} applies. By \autoref{lem:reasonable}, all but $O(\error^{1/7})n$ of the entries in the first row are reasonable. \autoref{lem:middle-value} implies that either all of the corresponding $\tau_{1i}$ are $26\error^{1/7}$-close to $\gamma$, or none are. In the latter case, they must be $26\error^{1/7}$-close to $\{0,1\}$, and so an argument similar to the preceding case proves the lemma. It remains to rule out the case that all reasonable $\tau_{1i}$ are $26\error^{1/7}$-close to $\gamma$.

Suppose for a contradiction that all the reasonable $\tau_{1i}$ are $26\error^{1/7}$-close to $\gamma$. Assume first that $c < 1/2$, so that $\gamma = 2c$. Since all but $O(\error^{1/7})n$ of the $\tau_{1i}$ are reasonable, we have
$$T=\sum_{i}\tau_{1i} \geq (1-O(\error^{1/7}))2cn.$$
This contradicts the equation $T = cn$ since $\error$ is small enough (not, here, depending on $c$).

Assume now that $c > 1/2$, so that $\gamma = 2c-1$. Then we have
$$(1-c)n = n-T= \sum_i (1-\tau_{1i}) \geq (1-O(\error^{1/7}))2(1-c)n,$$
a contradiction.
\end{proof}

Note that the assumption $\error^{1/7} = O(\eta)$ does not imply that $\gamma$ is $156\error^{1/7}$-far from $\{0,1\}$. Indeed, if $c$ is close to $1/2$ then $\gamma$ is close to $\{0,1\}$. Hence, it is necessary for us to split the proof of Lemma \ref{lem:large-entries:general} into the two cases above.

The analogue of \autoref{cor:main} now follows, just as before. We state it without any prior assumptions.

{
\renewcommand{\thecorollary}{\ref{cor:main}'}
\addtocounter{corollary}{-1}
\begin{corollary} \label{cor:main:general}
  Suppose that $n \geq 4$ and
 \begin{equation}
 \frac{1}{n^{7/3}} \leq \error \leq c_0 \eta^7, 
 \end{equation}
 where $c_0 > 0$ is an absolute constant. Let $\cF \subset S_n$ be a family of permutations with size $|\cF| = c \cdot n!$, satisfying
 \[ \EE[(f-f_1)^2] = \error, \]
where $f = 2\chi_\cF-1$, and $f_1$ is the orthogonal projection of $f$ onto $U_1$. Then there exists a family $\cG \subset S_n$ which is a union of $\lfloor cn \rfloor$ disjoint 1-cosets, satisfying
 \[ |\cG \triangle \cF| \leq O(\error^{1/7}) n!. \]
\end{corollary}
}

\paragraph{Getting rid of the assumptions on $\mathbf{\error}$} \autoref{cor:main:general} has a drawback: it needs to assume that $\error$ is not too small and not too large. When $\error$ is large enough (depending on $\eta$), the statement holds trivially, so we may focus our attention on the case where $\error$ is small. Intuitively, having $\error$ small should work in our favor. We shall introduce a few artificial errors to increase $\error$, and then later on take account of them, by introducing an extra error term into our conclusion statement.

We start by showing how to artificially increase $\error$.

\begin{lemma} \label{lem:error-introduction}
 Let $\cF \subset S_n$, and let $\upsilon \leq 1/16$. Then there exists a family $\cH \subset S_n$ such that
 \[
 |\cF \triangle \cH| \leq \sqrt{\upsilon} n! \quad \text{and}
 \quad \upsilon \leq \EE[(h-h_1)^2] \leq (\sqrt{\EE[(f-f_1)^2]} + 2\sqrt{\upsilon})^2,
 \]
where $f = 2\chi_\cF - 1$, $h = 2\chi_\cH - 1$ and $f_1,h_1$ are the projections of $f,h$ into $U_1$.

Moreover, if $|\cF| \geq n!/2$ then $\cH \subset \cF$, and otherwise $\cH \supset \cF$.
\end{lemma}
\begin{proof}
 By taking complements if necessary, we may assume that $|\cF| \geq n!/2$. Let $\sgn(\pi)$ denote the sign of a permutation $\pi$. Since $n \geq 3$, the sign function is orthogonal to $U_1$ (this is because the sign representation is not a constituent of the permutation representation), and so
 \begin{equation} \label{eq:error-lb}
 \EE[(f-f_1)^2] \geq \langle f,\sgn \rangle^2 = \left( \EE_\pi \sgn(\pi) f(\pi) \right)^2.
 \end{equation}
 First, assume that at least half of the permutations in $\cF$ are even. Then the number of these is at least $n!/4 \geq \sqrt{\upsilon} n!$. Define $\cG$ (and so $g$) by removing $\sqrt{\upsilon}n!$ of them. We have
 \[
  \EE_\pi \sgn(\pi) g(\pi) = \EE_\pi \sgn(\pi) f(\pi) - 2\sqrt{\upsilon}.
 \]
 Therefore, either $\EE_\pi \sgn(\pi) f(\pi) \geq \sqrt{\upsilon}$, or $\EE_\pi \sgn(\pi) g(\pi) \leq -\sqrt{\upsilon}$. In the former case, we take $\mathcal{H} = \mathcal{F}$, and we are done by the inequality~\eqref{eq:error-lb}. In the latter case, we take $\mathcal{H}=\mathcal{G}$. The inequality~\eqref{eq:error-lb} shows that $\EE[(h-h_1)^2] \geq \upsilon$. Moreover, since projections are contracting, we have
 \[
 \|h-h_1\|_2 \leq \|f-f_1\|_2 + \|(h-h_1)+(f-f_1)\|_2  \leq \|f-f_1\|_2 + \|h-f\|_2 \leq \|f-f_1\|_2 + 2\sqrt{\upsilon}.
 \]
 Similarly, if at least half of the permutations in $\cF$ are odd, then the number of these is at least $n!/4 \geq \sqrt{\upsilon} n!$. Define $\cG$ (and so $g$) by removing $\sqrt{\upsilon}n!$ of them. We have
 \[
  \EE_\pi \sgn(\pi) g(\pi) = \EE_\pi \sgn(\pi) f(\pi) + 2\sqrt{\upsilon}.
 \]
 Therefore, either $\EE_\pi \sgn(\pi) f(\pi) \leq -\sqrt{\upsilon}$, or $\EE_\pi \sgn(\pi) g(\pi) \geq \sqrt{\upsilon}$, so we may continue as before.
\end{proof}

Using this trick and \autoref{cor:main:general}, we get our main theorem in full generality.

\maintheorem*
\begin{proof}
 If $n < 4$ then the theorem is trivial (by taking the absolute constants implied by the $O$-terms to be sufficiently large), so we may assume that $n \geq 4$. If $\error$ satisfies~\eqref{eq:e-assumption:general}, then the theorem follows directly from \autoref{cor:main:general}. Otherwise, there are two cases: $\error$ is too large, and $\error$ is too small. If $\error > c_0\eta^7$ then the theorem holds, since $\error^{1/7}/\eta > c_0^{1/7}$, so suppose $\error < n^{-7/3}$.

Apply \autoref{lem:error-introduction} with $\upsilon = n^{-7/3}$ to obtain a family $\cH$. The value $\errorh = \EE[(h-h_1)]^2$ satisfies
\[
 \frac{1}{n^{7/3}} \leq \errorh \leq \left(\frac{1}{n^{7/6}} + \frac{2}{n^{7/6}}\right)^2 = \frac{9}{n^{7/3}}.
\]
Moreover, $|\cF \triangle \cH|/n! \leq n^{-7/6}$ and so $c_2 := \EE[\cH]$ satisfies $|c-c_2| \leq n^{-7/6}$. 
Also, $\eta_2 := \min(c_2,1-c_2)$ satisfies $|\eta-\eta_2| \leq n^{-7/6}$ as well.

There are two cases: either $\errorh > c_0\eta_2^7$, or not. In the first case, $9/n^{1/3} > c_0^{1/7}\eta_2$, and so $\eta_2 = O(n^{-1/3})$. Hence $\eta = O(n^{-1/3})$ and so the theorem holds, since $n^{-1/3}/\eta = \Omega(1)$.

The more interesting case is when $\errorh < c_0\eta_2^7$. \autoref{cor:main:general} applies, and we get a family $\cG \subset S_n$ which is the union of $\lfloor c_2n \rfloor = \lfloor c n \rfloor$ disjoint 1-cosets, where
 \[|\cG \triangle \cH| \leq O(\errorh^{1/7}) n!. \]
 Since $\errorh^{1/7} = O(n^{-1/3})$ and $|\cF \triangle \cH| \leq n^{-7/6} \cdot n!$, we have
$$|\cG \triangle \cF| \leq |\cG \triangle \cH|+ |\cH \triangle \cF| \leq O(\errorh^{1/7}) n! + n^{-7/6} \cdot n! \leq O(n^{-1/3})n!,$$
completing the proof of the theorem.
\end{proof}

\begin{remark} When $\error > c_0\eta^7$, the error terms $\error^{1/7}/\eta$ ensures that the theorem holds. Other error terms have the same effect, and so other versions of the theorem are possible. For example, instead of $\error^{1/7}/\eta$, we could have $\error^{1/7} + (\error^{1/7}/\eta)^2$.
\end{remark}

\section{Almost extremal isoperimetric sets in the transposition graph}\label{sec4}
As explained in the introduction, the main reason for developing Fourier-theoretic stability results, such as the main theorem of this paper, is for applications in extremal combinatorics. Oftentimes, one must struggle to translate the combinatorial information in an extremal problem to the Fourier language, but there is one setting in which the translation is almost immediate (yet may demand certain non-trivial calculations.) That is the setting of normal Cayley graphs on groups, and characterization of the maximum-sized independent sets, or the sets of minimum edge-expansion, in those graphs. See \cite{EFF1} for a more complete description of this. In a nutshell, there are good characterizations relating edge-expansion in graphs to the eigenvalues and eigenvectors of the graph, namely, the theorems of Alon-Milman \cite{AlonMilman} and Dodziuk \cite{Dodziuk}. A Cayley graph whose generating set is closed under conjugation is known as a {\em normal} Cayley graph. For any normal Cayley graph on a group $\Gamma$, its eigenspaces are precisely the {\em isotypical subspaces of $\mathbb{C}^\Gamma$} (the subspaces consisting of functions whose Fourier transform is concentrated on a fixed irreducible representation of $\Gamma$). Furthermore, the eigenvalues are given by a formula involving the average of the character of the corresponding representation on the generating set of the graph.

The example of the above phenomenon which we have in mind is the application of the Alon-Milman/Dodziuk theorems to the Cayley graph on $S_n$ generated by the transpositions. In other words, the graph $G$ with $V(G) = S_n$, and
$$E(G) = \{ \{\sigma,\tau\} : \sigma \tau^{-1}\ \mbox{is a transposition} \}$$
 --- two permutations are joined if they differ by a transposition. For any set $\mathcal{A} \subset V(G)$, we let $\partial A$ denote the edge-boundary of $\mathcal{A}$, i.e. the set of edges between $\mathcal{A}$ and its complement. As explained in \cite{EFF1},
by using Dodziuk/Alon-Milman, the work of Diaconis and Shashahani \cite{diaconis} yields the following theorem:
\begin{theorem}[Diaconis and Shashahani]
\label{thm:diaconis}
Let $\mathcal{A} \subset S_n$ with $|\mathcal{A}| = cn!$. Then
\begin{equation}\label{eq:diaconisbound}
| \partial \mathcal{A}| \geq (1-c)n |\mathcal{A}|, 
\end{equation}
with equality if and only if the characteristic vector of $\mathcal{A}$ belongs to $U_1$.
\end{theorem}
The characterization of Boolean functions in $U_1$ given in \cite{EFP} immediately yields the following characterization of the extremal isoperimetric sets.
\begin{corollary}
Let $\mathcal{A} \subset S_n$, with $|\mathcal{A}| = cn!$, and $| \partial \mathcal{A} | = (1-c)n |\mathcal{A}|$. Then $\mathcal{A}$ is a dictatorship.
\end{corollary}

We now want a stability version of this. In \cite{EFF1}, Lemma 13, we prove a stability version of Dodziuk/Alon-Milman, which, when combined with the eigenvalue estimates in \cite{diaconis}, shows that any set which has edge-boundary close to the minimum, must have its characteristic vector very close (in $L^2$ norm) to $U_1$:

\begin{theorem}[Lemma 13 in \cite{EFF1}]
 Let $\mathcal{A} \subset S_n$ with $|\mathcal{A}| = cn!$. If
$$
 | \partial \mathcal{A}| \leq (1-c+\delta_0)n |\mathcal{A}|,
$$
then 
$$
E[(f-f_1)^2] \leq \frac{n}{n-2}c\delta_0,
$$
where $f$ is the characteristic vector of $\mathcal{A}$, and $f_1$ is its projection on $U_1$.
\end{theorem}

Combining this with Theorem \ref{thm:main} immediately yields the following.
\begin{theorem}\label{thm:stableiso}
Let $\mathcal{A} \subset S_n$ with $|\mathcal{A}| = cn!$. If
$$
 | \partial \mathcal{A}| \leq (1-c+\delta_0)n |\mathcal{A}|
$$
then there exists a dictatorship $\mathcal{B} \subset S_n$ with 
$$
\frac{|\mathcal{A} \triangle \mathcal{B}|}{n!} = O\left(\frac{1}{c(1-c)} \left( (c\delta_0)^{1/7} + \frac{1}{n^{1/3}}\right)\right).
$$
\end{theorem}

We may apply a perturbation argument similar to the one in \cite{EFF1} to prove the following strengthening of Theorem \ref{thm:stableiso}:

\begin{theorem}
There exists $ \eta_0>0$ such that the following holds. Let $\mathcal{A} \subset S_n$ with $|\mathcal{A}| = cn!$, where $\min(c,1-c) \geq \eta_0.$ If
$$
 | \partial \mathcal{A}| \leq (1-c+\delta_0)n |\mathcal{A}|
$$
then there exists a dictatorship $\mathcal{B} \subset S_n$ with 
$$
\frac{|\mathcal{A} \triangle \mathcal{B}|}{n!} = O(\delta_0).
$$
\end{theorem}

This is best possible up to an absolute constant factor, as can be seen by taking
$$\mathcal{A} = T_{11} \cup T_{12} \cup \ldots \cup T_{1a} \cup T_{2,a+1} \cup T_{2,a+2} \cup \ldots \cup T_{2,a+b},$$
where $\min\{a/n,1-a/n\} = \Omega(1)$ and $b/a = \Theta(\delta_0)$.

Therefore, a subset of $S_n$ with measure bounded away from 0 and 1, which has edge-boundary close to the lower bound \ref{eq:diaconisbound}, must be close in structure to a dictatorship. This is a `genuine' stability result. One may contrast it with the `quasi-stability' result in \cite{EFF1}, where we prove that a subset of $S_n$ with size $\Theta((n-1)!)$ has edge-boundary close to the minimum iff it is close in structure to a {\em union} of dictatorships, as opposed to a single dictatorship.

\section{Conclusion}

\subsection*{Remarks and open questions} 

The most obvious open question in the context of this trilogy is whether it is possible to prove the common generalization of the main theorems in all three papers. Is it true that no matter what the expectation of $f$ is, if $f$ is Boolean and close to $U_t$, then it is close to a union of $t$-cosets? This surely must be true, but our techniques fall short of proving it. We also believe the correct dependence between the two distances to be linear. We make the following conjecture.

\begin{conjecture}
Let $\mathcal{A} \subset S_n$, and let $t \in \mathbb{N}$. Let $f$ denote the characteristic function of $\mathcal{A}$, and let $f_t$ denote the orthogonal projection of $f$ onto $U_t$. If
$$\mathbb{E}[(f-f_t)^2] \leq \epsilon \mathbb{E}[f],$$
then there exists a family $\mathcal{B} \subset S_n$ which is a union of $t$-cosets, such that
$$|\mathcal{A} \triangle \mathcal{B}|  \leq C_0(\epsilon+1/n) |\mathcal{A}|,$$
where $C_0$ is an absolute constant.
\end{conjecture}

Another related question involves understanding the precise extremal isoperimetric sets in the transposition graph on $S_n$, for all set-sizes. Limor Ben Efraim conjectures that the minimum edge-boundary is always achieved by an initial segment of the lexicographical order on $S_n$. (If $\sigma,\pi \in S_n$, we say that {\em \(\sigma < \pi\) in the lexicographic order} if \(\sigma(j) < \pi(j)\), where \(j = \min\{i \in [n]:\ \sigma(i) \neq \pi(i)\}\).)

It would also be interesting to discover other groups where there is an elegant characterization of Boolean functions whose Fourier support is concentrated on certain irreducible representations.

\subsubsection*{Acknowledgment} The authors wish to thank Gil Kalai for many useful conversations, and Caf\'e Karkur for its patient support of our research.


\begin{thebibliography}{99}
\bibitem{AlonMilman} N. Alon, V. D. Milman, $\lambda_1$, isoperimetric inequalities for graphs, and superconcentrators, {\em Journal of Combinatorial Theory, Series B}, 38 (1985), 73--88.

\bibitem{Berry} A. C. Berry, The accuracy of the Gaussian approximations to the sum of independent variates, \emph{Transactions of the American Mathematical Society} 49(1) (1941), 122--136.

\bibitem{Bourgain} J. Bourgain, On the distribution of the Fourier spectrum of Boolean
functions, \emph{Israel Journal of Mathematics} 131 (2002), 269--276.

\bibitem{diaconis} P. Diaconis, M. Shahshahani, Generating a random permutation with random transpositions, \emph{Z. Wahrsch. Verw. Gebeite}, Volume 57, Issue 2 (1981), 159--179.

\bibitem{Dodziuk} J. Dodziuk, Difference equations, isoperimetric inequality and transience of certain random walks, {\em Transactions of the American Mathematical Society} 284 (1984), 787--794.

\bibitem{tstability} D. Ellis, Stability for \(t\)-intersecting families of permutations, {\em Journal of Combinatorial Theory, Series A} 118 (2011), 208--227.

\bibitem{EFF1} D. Ellis, Y. Filmus, E. Friedgut, A quasi-stability result for dictatorships in $S_n$, {\em Combinatorica} 35 (2015), 573--618.

\bibitem{EFF3} D. Ellis, Y. Filmus, E. Friedgut, A quasi-stability result for low-degree Boolean functions on $S_n$, submitted. arXiv:1511.08694.

\bibitem{EFP} D. Ellis, E. Friedgut, H. Pilpel, Intersecting families of permutations, {\em Journal of the American Mathematical Society} 24 (2011), 649--682.

\bibitem{Esseen} C. G. Esseen, On the Liapunoff limit of error in the theory of probability, {\em Arkiv f\"or matematik, astronomi och fysik} A28 (1942), 1--19.

\bibitem{F-note} Y. Filmus, A comment on `Intersecting Families of Permutations', manuscript.\\
Available at \url{http://www.cs.toronto.edu/~yuvalf/EFP-comment.pdf}.

\bibitem{FriedgutJunta} {E. Friedgut}, Boolean functions with low average
sensitivity depend on few coordinates, \emph{Combinatorica} 18(1) (1998), 27--36.

\bibitem{FKN} E. Friedgut, G. Kalai, A. Naor, Boolean functions whose Fourier transform is concentrated on the first two levels, {\em Advances in Applied Mathematics} 29 (2002), 427--437.

\bibitem{KindlerODonnell}  G. Kindler, R. O'Donnell, Gaussian noise sensitivity and Fourier tails, \emph{IEEE Conference on Computational Complexity} (2012), 137--147.

\bibitem{KindlerSafra} G. Kindler, S. Safra, Noise resistant Boolean functions are juntas, manuscript.\\
Available at \url{http://www.cs.huji.ac.il/~gkindler/papers/noise-stable-r-juntas.ps.}

\bibitem{ProbIneq} Z. Lin, Z. Bai, {\em Probability inequalities}, \emph{Springer-Verlag}, Berlin, 2010.

\bibitem{NisanSzegedy}
N. Nisan, M. Szegedy, On the degree of Boolean functions as real polynomials,
{\em Computational Complexity} 4 (1994), 301--313.

\bibitem{sagan} B. E. Sagan, \emph{The Symmetric Group: Representations, Combinatorial Algorithms and Symmetric Functions}, Springer-Verlag, New York, 1991. [2nd revised printing, 2001.]

\bibitem{terras} A. Terras. {\em Fourier Analysis on Finite Groups and Applications.} London Mathematical Society Student Texts, 43, Cambridge University Press, 1999.

\end{thebibliography}
\end{document}